\newif\ifpictures
\picturestrue

\documentclass[12pt]{amsart}
\usepackage{amssymb,amsmath}
 \usepackage{amsopn}
 \usepackage{xspace}
  \usepackage{hyperref}
 \usepackage[dvips]{graphicx}
\usepackage[arrow,matrix,curve]{xy}
\usepackage {color, tikz}
\usepackage{wasysym}
\usepackage{subcaption}
\usepackage{mathtools} 

\usepackage{enumitem} 

\usepackage{tabularx}

\newcolumntype{L}[1]{>{\raggedright\arraybackslash}p{#1}}
\newcolumntype{C}[1]{>{\centering\arraybackslash}p{#1}}
\newcolumntype{R}[1]{>{\raggedleft\arraybackslash}p{#1}}

\numberwithin{equation}{section}
\newtheorem{thm}{Theorem}
\newtheorem{prop}[thm]{Proposition}
\newtheorem{lemma}[thm]{Lemma}
\newtheorem{cor}[thm]{Corollary}

\theoremstyle{definition}

\newtheorem{definition}[thm]{Definition}
\newtheorem{example}[thm]{Example}
\newtheorem{remark}[thm]{Remark}

\numberwithin{thm}{section}

\newcounter{FNC}[page]
\def\newfootnote#1{{\addtocounter{FNC}{2}$^\fnsymbol{FNC}$%
     \let\thefootnote\relax\footnotetext{$^\fnsymbol{FNC}$#1}}}

\newcommand{\N}{\mathbb{N}}

\newcommand{\R}{\mathbb{R}}

\renewcommand{\P}{\mathbb{P}}

\newcommand{\lf}{\left}
\newcommand{\ri}{\right}

\newcommand{\ovl}{\overline}

\newcommand\cV{{\ensuremath{\mathcal{V}}}\xspace}

\newcommand{\alp}{\alpha}
\newcommand{\lam}{\lambda}

\newcommand{\Sig}{\Sigma}

\newcommand{\struc}[1]{{\color{blue} #1}}

\DeclareMathOperator{\sgn}{sgn}
\DeclareMathOperator{\conv}{conv}

\DeclareMathOperator{\aff}{aff}

\DeclareMathOperator{\New}{New}

\DeclareMathOperator{\Int}{int}
\DeclareMathOperator{\Relint}{relint}

\DeclareMathOperator{\rank}{rank}

\def\endexa{\hfill$\hexagon$}

\tikzset{>=latex} 

\title{Real zeros of SONC polynomials}

\author{Mareike Dressler}

\address{Mareike Dressler, Department of Mathematics, University of California San Diego, 9500 Gilman Drive, La Jolla, California 92093, USA.\medskip}
\email{mdressler@ucsd.edu}


\begin{document}

	
\begin{abstract}
We provide a complete and explicit characterization of the real zeros of sums of nonnegative circuit (SONC) polynomials, a recent certificate for nonnegative polynomials independent of sums of squares. As a consequence, we derive an exact determination of the number $B''_{n+1,2d}$ for all $n$ and $d$. $B''_{n+1,2d}$ is defined to be the supremum of the number of zeros of all homogeneous $n+1$-variate polynomials of degree $2d$ in the SONC cone. The analogously defined numbers $B_{n+1,2d}$ and $B'_{n+1,2d}$ for the nonnegativity cone and the cone of sums of squares were first introduced and studied by Choi, Lam, and Reznick. In strong contrast to our case, the determination of both $B_{n+1,2d}$ and $B'_{n+1,2d}$ for general $n$ and $d$ is still an open question. 

Moreover, we initiate the study of the exposed faces of the SONC cone. In particular, we explicitly consider small dimensions and analyze dimension bounds on the exposed faces. When comparing the exposed faces of the SONC cone with those of the nonnegativity cone we observe dimensional differences between them.
\end{abstract}

\maketitle

\section{Introduction}
\label{Sec:Introduction}

One main goal in real algebraic geometry with crucial importance to polynomial optimization is to better understand the cone of nonnegative polynomials. Let \struc{$P_{n,2d}$} be the cone of nonnegative $n$-variate polynomials of degree at most $2d$. Since testing membership in $P_{n,2d}$ is very hard one is interested in finding and studying cones approximating $P_{n,2d}$, for which membership can be tested efficiently. A well-known example of such an inner approximation of the nonnegativity cone is \struc{$\Sig_{n,2d}$}, the cone of $n$-variate sums of squares (SOS) of degree at most $2d$. Their relationship has been studied since Hilbert's famous paper in 1888 \cite{Hilbert:Seminal}.   
In recent years, several other nonnegativity certificates have been proposed. For instance, Ahmadi and Majumdar considered the cones of diagonally\linebreak dominant sums of squares and scaled-diagonally sums of squares both lying inside $\Sig_{n,2d}$, see \cite{Ahmadi:Majumdar:DSOS&SDSOS}. Iliman and de Wolff \cite{Iliman:deWolff:Circuits} introduced the cone of \struc{\textit{sums of nonnegative circuit~(SONC)}} polynomials, which is independent of the SOS cone. SONC polynomials are a generalization of \struc{\textit{agiforms}} studied by Reznick \cite{Reznick:AGI}, whose nonnegativity can be certified via the arithmetic-geometric mean inequality. We denote the cone of all $n$-variate SONC polynomials of degree at most $2d$ by \struc{$C_{n,2d}$}. For a detailed definition of SONC polynomials and further details see Section \ref{SubSec:Prelim_SONC}. 

In the present paper, we analyze convex geometric aspects of the SONC cone. 
Studying convex geometric structures such as the boundary, the faces, and the dual cones of $P_{n,2d}$ and $\Sig_{n,2d}$ is an active area of research in convex algebraic geometry with many pending issues. For an overview of results, see \cite{Blekherman:Parrilo:Thomas}.

Analyzing these structures for the SONC cone is therefore naturally embedded in this evolving field.
Moreover, we hope that a better understanding of $C_{n,2d}$ and its convex geometric properties will also lead to an improvement from the practical point of view. Since SONC polynomials serve as a certificate of nonnegativity, which can be tested efficiently by geometric and relative entropy programs, they can be used in applications to polynomial optimization problems, as seen, e.g., in \cite{Iliman:deWolff:GP}, \cite{Dressler:Iliman:deWolff}, \cite{Dressler:Iliman:deWolff:REP}, and \cite{Dressler:Heuer:Naumann:deWolff:SONC_LP}. 

\vspace{0.4cm}
The primary focus of our work is on the study of the real zeros of SONC polynomials. 
Understanding the real zeros of polynomials is a research subject of intrinsic interest with a long and rich history and is especially useful for polynomials with certain properties like nonnegativity. Recall that the real zeros of nonnegative polynomials and sums of squares are in particular used to study the set theoretic difference $P_{n,2d}\setminus \Sig_{n,2d}$ and to construct explicit examples of nonnegative non-SOS forms. This started with Hilbert's above mentioned paper, \cite{Hilbert:Seminal}, which influenced many works on this subject, e.g., \cite{Choi:Lam:OldHilbertQuestion,Choi:Lam:ExtremalpsdForms,CLR:Knebusch:TransversalZerosandPSDForms, Reznick:SurveyHilbert17th,Kunert:Scheiderer:ExtremePositiveTernarySextics,Xu:Yao:Hilberts17forTernQuart}. 
Real zeros of nonnegative biquadratic forms are analyzed in  \cite{Quarez:RealZerosBiquadratric,Buckely:Sivic:NonnegativeBiquadraticForms}. We emphasize the work \cite{CLR:realzeros}, where Choi, Lam, and Reznick discuss nonnegative forms and properties of their real zeros, including their number.
Motivated by these former works, we provide an in-depth examination of the real zeros of SONC polynomials. Our main contribution is a complete and explicit characterization of the real zeros of both homogeneous and inhomogeneous SONC polynomials, see Section~\ref{Sec:RealZeros}. Among other results, we provide in Theorem~\ref{Thm:ExactNumberAffineZeroes} and Theorem~\ref{Thm:ExactNumberHomogZeroes} the explicit number of zeros of a nonnegative circuit polynomial in the affine and in the projective case. 

As an interesting consequence we obtain an exact determination of $B''_{n+1,2d}$ for all $n$ and $d$, see Theorem~\ref{Thm:B''}. \struc{$B''_{n+1,2d}$} is defined to be the supremum of the number of zeros of all homogeneous $n+1$-variate polynomials of degree $2d$ in the SONC cone, see Definition~\ref{Def:B''}. The analogously defined numbers $B_{n+1,2d}$ and $B'_{n+1,2d}$ for $\ovl P_{n,2d}$ and $\ovl \Sig_{n,2d}$ respectively were first introduced and studied by Choi, Lam, and Reznick in \cite{CLR:realzeros}. Here $\struc{\ovl P_{n+1,2d}}$ and $\struc{\ovl\Sig_{n+1,2d}}$ denote the cone of nonnegative forms and the cone of sums of squares of forms with $n+1$ variables of degree $2d$ respectively. In strong contrast to our case, the determination of both $B_{n+1,2d}$ and $B'_{n+1,2d}$ for general $n$ and $d$ is still an open question.

\vspace{0.4cm}
Beyond the study of the real zeros, the purpose of the paper is to initiate the study of the exposed faces of $C_{n,2d}$. A face $F$ is an \struc{\emph{exposed face}} of a convex set $S$ if there exists a supporting hyperplane $H$ to $S$ such that $F = S \cap H$.  

Analyzing the \emph{facial structures} of $\ovl P_{n+1,2d}$ and $\ovl \Sig_{n+1,2d}$ as well as the possible \emph{dimensions} of their faces is still an open problem. In the cases $n+1=2$, $2d=2$, and, to some extent, the case $(n+1,2d)=(3,4)$, thus, the Hilbert cases where $\ovl P_{n+1,2d} = \ovl \Sig_{n+1,2d}$, this problem is relatively well understood, see \cite{Barvinok:CourseConvexity}. But generally, only partial results are known. In particular, the facial analysis in the case where the two cones differ is important to better understand the gap between $\ovl P_{n+1,2d}$ and $\ovl \Sig_{n+1,2d}$.
Recently, Blekherman \cite{Blekherman:PSDandSOS} provided a geometric construction for the faces of the SOS cone that are not faces of $\ovl P_{n+1,2d}$, and in \cite{Blekherman:Iliman:Kubitzke:ExposedFaces} especially the dimensions of the exposed faces are investigated, this will be further discussed in Section~\ref{Sec:ExposedFaces}. 
The boundary of $\ovl P_{n+1,2d}$ consists of forms having a nontrivial zero, hence understanding the facial structure is closely related to the study of the real zeros of the cones.
Our contribution to this topic is an upper bound on the dimension of the exposed faces of $C_{n,2d}$, see Proposition~\ref{Prop:DimensionBoundExposedFaces}. We refine this bound in Proposition~\ref{Prop:DimensionBoundExposedFacesRefined}. In fact, in the univariate case, we state the dimension explicitly, see Lemma~\ref{Lem:DimExpFaceUnivariateCase}.

\medskip
The paper is structured as follows. Section~\ref{Sec:Preliminaries} fixes terminology, reviews basic concepts from convex geometry and about polynomials, and provides the theoretical background about real zeros. In Section~\ref{SubSec:Prelim_PSD-SOS} we briefly discuss interesting results about the real zeros of nonnegative polynomials and sums of squares. Section~\ref{SubSec:Prelim_SONC} introduces the main object of study, SONC polynomials. Afterward, in Section~\ref{Sec:AnalysisSONC} we present some structural results on the SONC cone. Section~\ref{Sec:RealZeros} is devoted to a comprehensive discussion on the real zeros of SONC polynomials. We add interesting observations as a result of the new knowledge about the real zeros in Section~\ref{SubSec:RealZeros_B-numbers}.  Finally, in Section~\ref{Sec:ExposedFaces}, we study the exposed faces of the SONC cone in small dimension.

%
\subsection*{Acknowledgments}
%
The results of this paper were mostly obtained during my PhD at the
Goethe University Frankfurt under Thorsten Theobald’s supervision. I am deeply grateful for his support and guidance throughout my PhD. I thank him also for valuable suggestions. I would like to further thank Timo de Wolff and an anonymous referee for helpful comments. 

\section{Preliminaries}
\label{Sec:Preliminaries}

First, we establish some notational conventions and introduce basic concepts from  convex geometry. Let $\struc{\N}=\{0,1,2,\ldots\}$ denote the set of nonnegative integers and \struc{$\R$} the field of real numbers. \struc{$\R_{>0}$} indicates the positive elements of $\R$. We also introduce the notation $\struc{\N^*}= \N  \setminus \{0\}$ and analogously $\struc{\R^*} = \R \setminus \{0\}$.
Throughout the paper bold letters denote  $n$-dimensional vectors unless noted otherwise. For a finite set $A \subset \N^n$ we denote by $\struc{\conv(A)}$ the \emph{convex hull} of $A$, and by $\struc{V(A)}$ the set of all the \emph{vertices} of $\conv(A)$. Similarly, we identify by $\struc{V(P)}$ the vertex set of any given polytope $P$. 
We call a lattice point $\boldsymbol{\alp} \in \N^n$ \struc{\emph{even}} if every entry $\alp_i$ is even, i.e., $\boldsymbol{\alp} \in (2\N)^n$. 
Furthermore, we denote by \struc{$\Delta_{n,2d}$} the \emph{standard simplex} in $n$ variables of edge length $2d$, i.e., the simplex satisfying $V(\Delta_{n,2d}) = \{\textbf{0},2d \cdot \mathbf{e}_1,\ldots,2d \cdot \mathbf{e}_n\}$.

Given a convex set $S \subset \R^n$, a \struc{\emph{face}} of $S$ is a subset $F\subseteq S$ such that for any point $p\in F$, whenever $p$ can be written as a convex combination of elements in $S$, these elements must belong to $F$. The \struc{\emph{dimension}} of a face $F$ is defined as the dimension of its affine hull, i.e., $\dim(F) \coloneqq \dim (\aff(F))$. 
We denote by $\struc{\Int(S)}$ ($\struc{\Relint(S)}$) the \emph{(relative) interior} and by $\struc{\partial S}$ the \emph{boundary} of $S$.

\subsection{Polynomials}
\label{SubSec:Prelim_Polynomials}

Let $\struc{\R[\mathbf{x}] = \R[x_1,\ldots,x_n]}$ be the ring of real $n$-variate polynomials. We usually consider \emph{polynomials} $f \in \R[\mathbf{x}]$ \textit{supported} on a finite set $A \subset \N^n$. Thus, $f$ is of the form $\struc{f(\mathbf{x})} = \sum_{\boldsymbol{\alp} \in A}^{} f_{\boldsymbol{\alp}}\mathbf{x}^{\boldsymbol{\alp}}$ with $f_{\boldsymbol{\alp}} \in \R$ and the \emph{monomial} $\mathbf{x}^{\boldsymbol{\alp}} = x_1^{\alp_1} \cdots x_n^{\alp_n}$ whose \emph{degree} is $\struc{|\boldsymbol{\alp}|}=\sum_{i=1}^{n} \alp_i$. The degree of the polynomial $f$ is given by the maximum degree over all appearing monomials, i.e., $\struc{\deg(f)} = \max \{|\boldsymbol{\alp}|: f_{\boldsymbol{\alp}}\neq 0\}$. 
We say that a polynomial is a \struc{\emph{sum of monomial squares}} if all terms $f_{\boldsymbol{\alp}}\mathbf{x}^{\boldsymbol{\alp}}$ satisfy $f_{\boldsymbol{\alp}} > 0$ and $\boldsymbol{\alp}$ is even. 
The set of all $n$-variate polynomials of degree less than or equal to $2d$ is denoted by $\struc{\R[\mathbf{x}]_{n,2d}}$. 
The \emph{Newton polytope} of a polynomial $f$ is defined as $\struc{\New(f)} \coloneqq \conv\{\boldsymbol{\alp} \in A : f_{\boldsymbol{\alp}} \neq 0\}$.   

A polynomial in which all terms are of the same degree is called a \emph{\struc{homogeneous polynomial}} or a \struc{\emph{form}}. If $f \in \R[\mathbf{x}]_{n,2d}$ is any polynomial, then
\begin{align*}
\struc{\ovl f (x_0,\ldots,x_n)}\;= \; x_0^{2d} \, f\lf(\frac{x_1}{x_0}, \ldots, \frac{x_n}{x_0}\ri)
\end{align*}
is the \emph{homogenization} of $f\!$, which is a form of degree $2d$ in the $n+1$ variables $x_0,x_1,\ldots,x_n$. Given a form $\ovl f $ we can \emph{dehomogenize} it by setting $x_0=1$. Since we are switching back and forth between both viewpoints,  for a better distinguishability, we fix the above notation and always write polynomials as $f \in \R[\mathbf{x}]_{n,2d}$ and forms as $\ovl f \in \R[x_0,\mathbf{x}]_{n+1,2d}$. 
Whereby the form $\ovl f$ can also be considered as a function on the real projective $n$-space $\struc{\P^n}$.

Finally, we define the \struc{\emph{zero-set}} of a polynomial $f$ and of a form $\ovl f$ respectively, by
\begin{align*}
\struc{\cV(f)}&\; \coloneqq\; \{(v_1,\ldots,v_n) \in \R^{n}: f(v_1,\ldots,v_n)=0 \}, \\
\struc{\cV(\overline{f})}&\; \coloneqq \; \lf\{[v_0:\cdots:v_n] \in \P^{n} : \overline{f}(v_0,\ldots,v_n)=0 \ri\}.
\end{align*}
We denote by $\struc{|\cV(\cdot)|}$ the number of distinct elements in the zero-set. 
The zero-set of a form may be viewed as the set   
\begin{align*}
\cV(\overline{f}) \; =\; \{(v_0,\ldots,v_n) \in \R^{n+1}\setminus \{0,\boldsymbol{0}\} : \overline{f}(v_0,\ldots,v_n)=0 \},
\end{align*}
where $|\cV(\overline{f})|$ will be interpreted as the number of lines in $\cV(\overline{f})$ and we only count one representative of each line.\\
In a natural way, there may occur \struc{\emph{zeros of $f$ at infinity}} via homogenization. This is the case if $v_0=0$ for $(v_0,\boldsymbol{v}) \in \cV(\overline{f})$. If $v_0 \neq 0$, then $(v_0,\boldsymbol{v})$~corresponds~to~a~unique~zero~of~$f$.

\subsection{Positive Polynomials and Sums of Squares}
\label{SubSec:Prelim_PSD-SOS}

In what follows let $\struc{P_{n,2d}}$ be the cone of nonnegative $n$-variate polynomials of degree at most $2d$ and let $\struc{\Sig_{n,2d}}$ denote the cone of $n$-variate sums of squares of degree at most $2d$. With $\struc{\ovl P_{n+1,2d}}$ and $\struc{\ovl\Sig_{n+1,2d}}$ we mean the cone of nonnegative forms and the cone of sums of squares of forms in $\R[x_0,\mathbf{x}]_{n+1,2d}$ respectively.

In his seminal paper Hilbert \cite{Hilbert:Seminal} classified all cases in which these two cones coincide:
\begin{thm}[Hilbert, 1888]
	Let $P_{n,2d}$ and $\Sigma_{n,2d}$ be as explained, then \\
	$P_{n,2d} = \Sigma_{n,2d}$ \quad if and only if \quad $n=1$ \;or\; $d=2$ \;or\; $(n,2d)=(2,4)$.	
	\label{Thm:HilbertConeCoincidePSOS}
\end{thm}

The investigation of the real zeros of $\ovl P_{n+1,2d}$ and $\ovl \Sig_{n+1,2d}$ is part of active research in convex algebraic geometry. Choi, Lam, and Reznick study in \cite{CLR:realzeros} the real zeros of nonnegative forms and provide various consequential results. 
%
%
Therein, the authors also define certain numbers given by the maximal number of real zeros of forms in the SOS and the nonnegativity cone:
\[ \struc{B_{n+1,2d}}\; \coloneqq\; \sup_{\substack{\ovl f\in \ovl P_{n+1,2d}\\|\cV(\ovl f)| < \infty}} |\cV(\ovl f)|\qquad\textrm{and}\qquad \struc{B_{n+1,2d}'} \;\coloneqq\; \sup_{\substack{\ovl f\in \ovl\Sig_{n+1,2d}\\|\cV(\ovl f)| < \infty}} |\cV(\ovl f)|. \]
Geometrically, $B_{n+1,2d}$ corresponds to the largest size of a finite real projective hypersurface of degree $2d$ in $\P^n$.
To determine these numbers exactly is a rather difficult task. In what follows, we list some known results in a few special cases. 
\begin{thm}[\cite{Blekherman:et:al:Boundaries, CLR:realzeros, Shafarevich:BAG}]
	\label{Thm:B-numbersSOS&PSD}
	Let $B_{n+1,2d}$ and $B_{n+1,2d}'$ be as defined above, then:
	\renewcommand{\labelenumi}{\textnormal{(\arabic{enumi})}}
	\begin{enumerate}
		\item $B_{2,2d} = B'_{2,2d} = d$ and $B_{n+1,2}=B'_{n+1,2}=1$.
		\item $B_{3,4}=4$ and $B_{3,6} = B_{4,4} = 10$.
		\item  $d^2 \leq B_{3,2d} \leq \frac{3d(d - 1)}{2} + 1$ for $2d \geq 6$, and
		\item $B_{3,6k} \geq 10k^2$, $B_{3,6k+2} \geq 10k^2 + 1$, $B_{3,6k+4} \geq 10k^2 + 4$.
		\item Let $\beta(2d) = \frac{B_{3,2d}}{4d^2}$. Then $\beta=\lim_{2d \to \infty}  \beta(2d)$ exists. Moreover, $\beta(2d)\leq \beta$ for all $\frac{5}{18}\leq \beta \leq \frac{1}{2}$.
		\item $B'_{n+1,2d} \geq d^{n}$.
		\item $B_{3,2d}' = \frac{(2d)^2}{4}=d^2$.
	\end{enumerate}
\end{thm}

Theorem~\ref{Thm:B-numbersSOS&PSD} shows that $B_{3,2d}$ is always finite. For quartic forms we have $B_{4,4}=10$, but already for $n+1\geq 5$, we do not know if $B_{n+1,4}$ needs to be finite in general.

\subsection{Sums of Nonnegative Circuit Polynomials}
\label{SubSec:Prelim_SONC}

In this subsection we introduce SONC polynomials and their basic properties, which are used in this work. SONC polynomials are composed of \textit{circuit polynomials}, see \cite{Iliman:deWolff:Circuits}. 
\begin{definition}
	Let $f \in \R[\mathbf{x}]$ be supported on $A \subset \N^n$. Then $f$ is called a \struc{\emph{circuit polynomial}} if it is of the form
	\begin{eqnarray}
	\struc{f(\mathbf{x})} & = & \sum_{j=0}^r f_{\boldsymbol{\alp}(j)} \mathbf{x}^{\boldsymbol{\alp}(j)} + f_{\boldsymbol{\beta}} \mathbf{x}^{\boldsymbol{\beta}}, \label{Equ:CircuitPolynomial}
	\end{eqnarray}
	with $\struc{r} \leq n$, exponents $\struc{\boldsymbol{\alp}(j)}$, $\struc{\boldsymbol{\beta}} \in A$, and coefficients $\struc{f_{\boldsymbol{\alp}(j)}} \in \R_{> 0}$, $\struc{f_{\boldsymbol{\beta}}} \in \R$, such that the following conditions hold:
	
	\begin{description}
		\item[(C1)] $\New(f)$ is a simplex with even vertices $\boldsymbol{\alp}(0), \boldsymbol{\alp}(1),\ldots,\boldsymbol{\alp}(r)$. 
		\item[(C2)] 
		The exponent $\boldsymbol{\beta}$ can be written uniquely as 
		\begin{eqnarray*}
			& & \boldsymbol{\beta} \ = \ \sum_{j=0}^r \lambda_j \boldsymbol{\alp}(j) \ \text{ with } \ \lambda_j \ > \ 0 \ \text{ and } \  \sum_{j=0}^r \lambda_j \ = \ 1
		\end{eqnarray*}
		in \struc{\emph{barycentric coordinates} $\lambda_j$} relative to the vertices $\boldsymbol{\alp}(j)$ with $j=0,\ldots,r$.
	\end{description}
	We call the terms $f_{\boldsymbol{\alp}(0)} \mathbf{x}^{\boldsymbol{\alp}(0)},\ldots,f_{\boldsymbol{\alp}(r)} \mathbf{x}^{\boldsymbol{\alp}(r)}$ the \struc{\emph{outer terms}} and $f_{\boldsymbol{\beta}} \mathbf{x}^{\boldsymbol{\beta}}$ the \struc{\emph{inner term}} of $f$. For the corresponding exponents we refer to the \struc{\emph{outer exponents}} and the \struc{\emph{inner exponent}} of $f$ respectively.
	
	Every circuit polynomial determines a certain invariant, the \struc{\textit{circuit number}}, via
{
	\setlength{\belowdisplayskip}{-0.2cm}
	\begin{eqnarray}
	\struc{\Theta_f} \ = \ \prod_{j = 0}^r \left(\frac{f_{\boldsymbol{\alp}(j)}}{\lambda_j}\right)^{\lambda_j}. \label{Equ:DefCircuitNumber}
	\end{eqnarray}
}
	\label{Def:CircuitPolynomial}
	\endexa
\end{definition}

Observe that Condition (C2) implies, that the inner exponent $\boldsymbol{\beta}$ is in the strict interior of $\New(f)$ if $\dim(\New(f))\geq 1$. For $r=0$ we have $\boldsymbol{\beta}=\boldsymbol{\alp}(0)$, i.e., in this case $f$ is a monomial square.  \\

The terms ``circuit polynomial'' and ``circuit number'' are chosen since the support $A=\{\boldsymbol{\alp}(0), \boldsymbol{\alp}(1),\ldots,\boldsymbol{\alp}(r),\boldsymbol{\beta}\}$ forms a \struc{\textit{circuit}}, this is a minimally affine dependent set, see e.g. \cite{Gelfand:Kapranov:Zelevinsky}.

A fundamental fact is that nonnegativity of a circuit polynomial $f$ can be decided easily by its circuit number $\Theta_f$ alone.

\begin{thm}[\cite{Iliman:deWolff:Circuits}, Theorem 3.8]
	Let $f$  be a  circuit polynomial as in Definition~$\ref{Def:CircuitPolynomial}$. Then the following are equivalent:
	\begin{enumerate}
		\item $f$ is nonnegative.
		\item $|f_{\boldsymbol{\beta}}| \leq \Theta_f$ and $\boldsymbol{\beta} \not \in (2\N)^n$ \quad or \quad $f_{\boldsymbol{\beta}} \geq -\Theta_f$ and $\boldsymbol{\beta }\in (2\N)^n$.
	\end{enumerate}
	\label{Thm:CircuitPolynomialNonnegativity}
\end{thm}

An immediate consequence that can be drawn by the proof of the above theorem is an upper bound for the number of zeros of circuit polynomials and a condition for a circuit polynomial to lie on the boundary of the cone of nonnegative polynomials, but we  postpone this discussion to Section~\ref{Sec:RealZeros}.

Writing a polynomial as a sum of nonnegative circuit polynomials is a certificate of nonnegativity. We denote by \struc{SONC} both the class of polynomials that are \struc{\it sums of nonnegative circuit polynomials} and the property of a polynomial to be in this class.  

\begin{definition}
For every $n,d \in \N^*$ the set of \struc{\emph{sums of nonnegative circuit polynomials} (SONC)} in $n$ variables of degree $2d$ is defined as
	$$\struc{C_{n,2d}} \ := \ \left\{p \in \R[\mathbf{x}]_{n,2d} \ :\  p = \sum_{i=1}^k \mu_i f_i, \begin{array}{c}
	f_i \text{ is a nonnegative circuit polynomial, } \\
	\mu_i \geq 0, k \in \N^* \\
	\end{array}
	\right\}.
	$$
	\label{Def:SONC}
	\endexa
\end{definition}

Indeed, SONC polynomials form a convex cone independent of the SOS cone.

\begin{thm}[\cite{Iliman:deWolff:Circuits}, Proposition 7.2]
	$C_{n,2d}$ is a convex cone satisfying:
	\begin{enumerate}
		\item $C_{n,2d} \subseteq P_{n,2d}$ for all $n,d \in \N^*$,
		\item $C_{n,2d} \subseteq \Sigma_{n,2d}$ if and only if $(n,2d)\in\{(1,2d),(n,2),(2,4)\}$,
		\item  $\Sigma_{n,2d} \not\subseteq C_{n,2d}$ for all $(n,2d)$ with $2d \geq 6$.
	\end{enumerate}
	\label{Thm:ConeContainment}
\end{thm}
\hspace*{-0.25cm} For further details about the SONC cone see, e.g.,  \cite{Iliman:deWolff:Circuits,Dressler:Iliman:deWolff,Dressler:Iliman:deWolff:REP,Dressler:Kurpisz:deWolff:SONCHypercube}. In recent works, amongst others, the dual of the SONC cone, its algebraic boundary, and extreme rays are studied, see \cite{Dressler:Naumann:Theobald:SONCdual,Katthaen:Naumann:Theobald:UnifiedFramework, Forsgard:deWolff:SONC-Boundary}. The SONC cone closely relates to the `SAGE' cone established by Chandrasekaran and Shah in \cite{Chandrasekaran:Shah:SAGE} and further studied and generalized by Murray, Chandrasekaran, and Wierman \cite{Murray:Chandrasekaran:Wiermann:Newton2018,Murray:Chandrasekaran:Wiermann:SigPolyOpt2019}. We refer readers who are interested in the relation of these cones to the discussions in \cite[Section 2.1]{Dressler:Heuer:Naumann:deWolff:SONC_LP} and \cite[Introduction]{Murray:Naumann:Theobald:2020xcircuits}.

\medskip

\section{Deeper Analysis of the SONC cone}
\label{Sec:AnalysisSONC}

In this section we present some important properties and observations of the SONC cone. First, we provide the missing case of the statement about the \mbox{(non-)}containment of the SONC cone and the SOS cone. Afterward, we consider the realizability of nonnegative circuit polynomials with a certain degree. We conclude this section by proving that (de-) homogenizing a SONC polynomial stays SONC.

\hspace*{-0.23cm}The subsequent statement covers the missing case of Theorem~\ref{Thm:ConeContainment} (3), where \mbox{$\Sigma_{n,2d} \not\subseteq C_{n,2d}$} is only shown for degree $2d \geq 6$, answering a question stated in \cite{Iliman:deWolff:Circuits}. Thus, we get the following result, whereby we actually prove the full statement rather than only the missing cases.   

\begin{thm}
	$\Sigma_{1,2} = C_{1,2}$, $\Sigma_{n,2} \not\subseteq C_{n,2}$ for all $n\geq 2$, and $\Sigma_{n,2d} \not\subseteq C_{n,2d}$ for all $(n, 2d)$ with $2d \geq 4$.
	\label{Thm:MissingPieceConeContainment}	
\end{thm}

\begin{proof}
The first statement was already observed in \cite{Iliman:deWolff:Circuits}, but not proven as it is rather obvious. For the sake of completeness, we give a short argument here.
Consider a polynomial $p \in \R[x]_{1,2}$, i.e., $p(x)=ax^2+bx+c$, with $a,b,c \in \R$. Note that the support of $p$ forms a circuit. Obviously, $p$ is nonnegative if and only if $p$ is a nonnegative circuit polynomial. Hence, $C_{1,2}=P_{1,2}=\Sigma_{1,2}$, where the last equality follows by Hilbert's Theorem~\ref{Thm:HilbertConeCoincidePSOS}.

	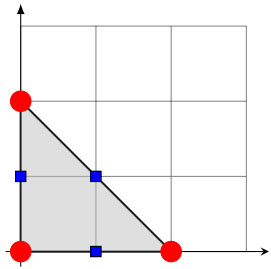
\begin{figure}
	
	\centering
	\begin{subfigure}[t]{0.5\linewidth}
		\centering
		\begin{tikzpicture}
		\draw[very thin,color=gray] (0,0) grid (3,3);    
		\draw[->, >=stealth] (-0.2,0) -- (3.3,0) node[right]{}; 
		\draw[->] (0,-0.2) -- (0,3.3) node[above]{}; 
		
		\coordinate (A) at (0,0);
		\coordinate (B) at (2,0);
		\coordinate (C) at (0,2);
		
		\draw[thick] (A) -- (B);
		\draw[thick] (C) -- (B);
		\draw[thick] (A) -- (C);
		
		\draw[fill=lightgray, opacity=0.5] (A) -- (B) -- (C);
		\node[circle, draw=red, fill=red, scale=0.7] (A) at (0,0) {};
		\node[circle, draw=red, fill=red, scale=0.7] (B) at (2,0) {};
		\node[circle, draw=red, fill=red, scale=0.7] (C) at (0,2) {};
		
		\node[draw, fill=blue, scale=0.5](D) at (1,1) {};
		\node[draw, fill=blue, scale=0.5](E) at (1,0) {};
		\node[draw, fill=blue, scale=0.5](F) at (0,1) {};
		
		\end{tikzpicture}
	\end{subfigure}%
%
%
%
%
%
%
	\caption{The support set of $q(x_1,x_2)$.
		 The even points are the red (round) ones.}
	\label{Fig:CounterexampleSOSnotinSONC:supportofqandr}	
\end{figure}

For proving the second assertion, we explicitly construct a polynomial which is SOS but not SONC. First we explain the bivariate case in detail, then we generalize this idea to an  arbitrary number of variables. 
Consider the following bivariate polynomial of degree $2$
\[
q(x_1,x_2)= 1 + x_1^2 + x_2^2 + 2x_1x_2+2x_1+2x_2=(1+x_1+x_2)^2,
\]
which is clearly SOS. 
The Newton polytope of $q$ is the standard simplex $\Delta_{2,2}$. Additionally to the even vertices we have odd support points on every edge of $\Delta_{2,2}$, see Figure~\ref{Fig:CounterexampleSOSnotinSONC:supportofqandr}. Therefore, the only possibility to write $q$ as a sum of circuit polynomials is:
\[
q= f_1+f_2+f_3= \lf(\frac{1}{2}+\frac{1}{2}x_1^2+2x_1\ri)+\lf(\frac{1}{2}+\frac{1}{2}x_2^2+2x_2\ri)+\lf(\frac{1}{2}x_1^2+\frac{1}{2}x_2^2+2x_1x_2\ri).
\]
Clearly, all circuit polynomials $f_i$ are not nonnegative, because $2 > \Theta_{f_i}=1$ for all $i \in \{1,2,3\}$. Thus, $q$ is not a SONC polynomial. 
For $n\geq 2$ we generalize this idea by constructing a polynomial whose support consists of the vertices of the standard simplex $\Delta_{n,2}$ and, in addition, the midpoints of each of its $\tbinom{n}{2}$ edges:
\[
\hat q(\mathbf{x})=(1+x_1+x_2+\cdots+x_n)^2.
\]
Showing that $\hat q$ is not a SONC polynomial is analogous to the bivariate case.


To prove the third statement, let $f$ be a non-zero nonnegative circuit polynomial in one variable. We first claim that if there is some $v\in \R$ such that $f(v)=f'(v)=0$, then $f''(v)>0$, that is, $f$ may have at most second order zeros. 
It then follows immediately that $(1+x)^{2d}$ for $2d\geq 4$ cannot be a SONC polynomial. Transferring this idea to further variables yields a polynomial showing the general multivariate case for $2d \geq 4$:
\begin{align}
\label{Equ:Polynom_WitnessingMissingCase}
(1+x_1)^{2d} + \cdots + (1+x_n)^{2d} \;\in \Sig_{n,2d}\setminus C_{n,2d} .
\end{align}
To prove the claim, let $p(x)=f(x/v), v\neq 0$. Clearly, $p$ is a nonnegative circuit polynomial as well and $p(1)=p'(1)=0$ by hypothesis. Let $p(x)=ax^n+bx^r+c$, with $0<r<n$. Then we have $a+b+c=an+br=0$, so $b=-\frac{n}{r}a$. If $p''(1)=0$ as well, then $an(n-1)+br(r-1)=an(n-r)=0$, hence $a=0$ and thus, $p=0$ contradicting our assumption. Therefore, $p''(1)\neq 0$ and since $p$ is nonnegative, $p''(1)>0$. 
\end{proof}	

An immediate consequence of the proof of the third statement of Theorem~\ref{Thm:MissingPieceConeContainment}  is the subsequent interesting result.
\begin{cor}
	Squares of SONC polynomials are in general not SONC polynomials.
\end{cor}
\begin{proof}
	The result can be deduced by the fact, that the Hessian of a SONC polynomial is positive definite at its zeros.
\end{proof}
\begin{remark}
Choosing $2d=4$ in \eqref{Equ:Polynom_WitnessingMissingCase} yields a polynomial in $\Sig_{n,4}\setminus C_{n,4}$, which explicitly witnesses the missing case of Theorem~\ref{Thm:ConeContainment} (3).
\end{remark}
\smallskip
Note that the degree of each variable part in the sum of \eqref{Equ:Polynom_WitnessingMissingCase} needs not to be the same, i.e., a more general polynomial in $\Sig_{n,2d}\setminus C_{n,2d}$ is
\[
(1+x_1)^{2d_1} + \cdots + (1+x_n)^{2d_n},
\]
where $d\coloneqq \max\{d_1,\ldots, d_n\}$ and $d_i\geq 2$ for all $i=1,\ldots,n$.

\smallskip
Another reasonable approach to find a polynomial showing $\Sigma_{n,2d} \not\subseteq C_{n,2d}$ for all $(n, 2d)$ with $n\geq 2$ and $2d \geq 4$ is by making use of the idea in the proof of the second assertion of Theorem~\ref{Thm:MissingPieceConeContainment}. Namely, to construct a polynomial whose support contains points on the boundary of the Newton polytope. More precisely, we reflect the standard simplex $\Delta_{2,2}$ with additional boundary points on each edge with respect to the hypotenuse, see Figure~\ref{Fig:CounterexampleSOSnotinSONC_supports}, to get the following bivariate polynomial of degree $4$:
\[
s(x_1,x_2)=(x_1+x_2+x_1x_2)^2=x_1^2x_2^2+x_1^2+x_2^2+2x_1x_2+2x_1^2x_2^{}+2x_1^{}x_2^2.
\]
This polynomial is SOS but not SONC. By adding additional variables of certain degree,
yields a polynomial in $\Sig_{n,2d}\setminus C_{n,2d}$:
\begin{align*}
\hat s(\mathbf{x})=(x_1+x_2+x_1x_2)^2 + \sum_{i=3}^{n} x_i^{2d_i},
\end{align*}
with $d\coloneqq\max\{d_1,\ldots, d_n\}$ and $d_i\geq 2$ for all $i=1,\ldots,n$.

\begin{figure}[h]
	\centering
	\begin{tikzpicture}
	\draw[very thin,color=gray] (0,0) grid (3,3);    
	\draw[->] (-0.2,0) -- (3.3,0) node[right]{}; 
	\draw[->] (0,-0.2) -- (0,3.3) node[above]{}; 
	
	\coordinate (A) at (2,2);
	\coordinate (B) at (2,0);
	\coordinate (C) at (0,2);
	
	\draw[thick] (A) -- (B);
	\draw[thick] (C) -- (B);
	\draw[thick] (A) -- (C);
	
	\draw[fill=lightgray, opacity=0.5] (A) -- (B) -- (C);
	
	\node[circle, draw=red, fill=red, scale=0.7] (A) at (2,2) {};
	\node[circle, draw=red, fill=red, scale=0.7] (B) at (2,0) {};
	\node[circle, draw=red, fill=red, scale=0.7] (C) at (0,2) {};
	
	\node[draw, fill=blue, scale=0.5](D) at (1,1) {};
	\node[draw, fill=blue, scale=0.5](E) at (1,2) {};
	\node[draw, fill=blue, scale=0.5](F) at (2,1) {};
	
	\end{tikzpicture}
	\caption{The support set of $s(x_1,x_2)$.}
	\label{Fig:CounterexampleSOSnotinSONC_supports}		
\end{figure}
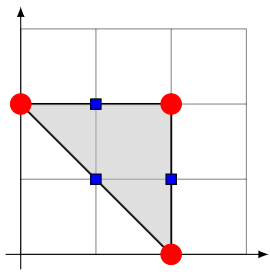

Given a circuit polynomial $f$ we say $f$ is \struc{\emph{proper}} if it is not a sum of monomial squares.
Due to the special structure of circuit polynomials, we make the following simple observation.


\begin{lemma}
	Let $f$  be an $n$-variate circuit polynomial of degree $2d$ with inner term $f_{\boldsymbol{\beta}} \mathbf{x}^{\boldsymbol{\beta}}$. If $f$ is a proper circuit polynomial then
	\renewcommand{\labelenumi}{\textnormal{(\arabic{enumi})}}
	\begin{enumerate}
		\item $2d \geq n+1$, for $\boldsymbol{\beta} \notin (2\N)^n$,
		\item $d\geq n+1$, for $\boldsymbol{\beta} \in (2\N)^n$.
	\end{enumerate}
	\label{Lem:RealizabilityCP}
\end{lemma}

\begin{proof}	
	Consider a circuit polynomial $f$ as in \eqref{Equ:CircuitPolynomial}
	\[
	f(x_1,\ldots,x_n) =  \sum_{j=0}^r f_{\boldsymbol{\alp}(j)} \mathbf{x}^{\boldsymbol{\alp}(j)} + f_{\boldsymbol{\beta}} \mathbf{x}^{\boldsymbol{\beta}}.
	\]
	Recall that $\boldsymbol{\beta} \in \Relint(\New(f))$. Thus, the smallest possible inner exponent $\boldsymbol{\beta}$ such that $f$ may be a circuit polynomial is $\beta_i=1$, $i=1,\ldots,n$ for $\boldsymbol{\beta}$ odd and $\beta_i=2$,  $i=1,\ldots,n$ if $\boldsymbol{\beta}$ is even. It follows immediately that if the number of variables is strictly smaller than the degree, i.e., $n < 2d$, a circuit polynomial with odd inner exponent is realizable. And a circuit polynomial with $\boldsymbol{\beta} \in (2\N)^n$ is only realizable if $2n < 2d$. 
\end{proof}

Lemma~\ref{Lem:RealizabilityCP} holds of course as well for nonnegative circuit polynomials. The above statement is in particular useful for studying the real zeros of SONC polynomials.

\medskip
So far SONC polynomials are only studied in the affine case. In the literature often results on nonnegative polynomials and SOS are stated homogeneously, i.e., for forms. In what follows we also consider SONC forms and investigate their behavior. A first important observation is, that the property to be SONC is inherited under homogenization and, conversely, is preserved when a form is dehomogenized into a polynomial.
\begin{prop}
	\label{Prop:SONC homo preserved}
	If a polynomial $p \in \R[x_1,\ldots,x_n]$ is a SONC polynomial of degree $2d$, then its homogenization $\ovl p (x_0,\ldots,x_n)= x_0^{2d} \, p\lf(\frac{x_1}{x_0}, \ldots, \frac{x_n}{x_0}\ri)$ is also SONC; and vice versa.
\end{prop}

\begin{proof}
Since a SONC polynomial $p$ is sum of nonnegative circuit polynomials $f_i$, i.e., $p = \sum_{i=1}^k \mu_i f_i, \mu_i \geq 0$, it suffices to prove the statement for circuit polynomials. Note that the $f_i$ may have different degrees. Hence, homogenize every $f_i$ with $x_0^{2d}$ even though its degree may be smaller.

For a single circuit polynomial $f$ the statement is clear because in this case, the SONC cone is equal to the nonnegativity cone, i.e., $f\geq 0$ if and only if $\ovl f \geq 0$. 
\end{proof}	

Even more is true, namely the circuit numbers of a nonnegative circuit polynomial and its homogenization are equal. 
\begin{lemma}
Let $f\in C_{n,2d}$ be a nonnegative circuit polynomial and $\ovl f$ be its homogenization, then $\Theta_f = \Theta_{\ovl f}$.
\end{lemma} 

\begin{proof}
Recall that the barycentric coordinates are given by the convex combination of the interior points in terms of the vertices, thus by a system of $n$ resp. $n+1$ linear equations in $r$ unknowns:
	$$\boldsymbol{\beta} = \sum_{j=0}^r \lambda_j \boldsymbol{\alp}(j)\; \text{ and }\; (2d-|\boldsymbol{\beta}|,\boldsymbol{\beta})^T = \sum_{j=0}^r \ovl\lambda_j (2d-|\boldsymbol{\alp}(j)|,\boldsymbol{\alp}(j))^T. $$
Obviously, the row given by the homogenization is  linearly dependent from the other rows, hence $\lambda _j= \ovl \lambda_j$ for all $j=0,\ldots,r$. Thus, $\Theta_f = \Theta_{\ovl f}$. 
\end{proof}
\medskip

\section{Real Zeros of SONC Polynomials}
\label{Sec:RealZeros}

This section offers a comprehensive classification of the real zeros of SONC polynomials as well as SONC forms. 

A first result for an upper bound of affine real zeros for a nonnegative circuit polynomial having a constant term is given by Iliman and de Wolff \cite[Corollary 3.9]{Iliman:deWolff:Circuits} as a consequence of Theorem~\ref{Thm:CircuitPolynomialNonnegativity}. Here, we study the more general case of circuit polynomials, which do not need to have a constant term. If we want to emphasize that a circuit polynomial does not have a constant term, we refer to a \struc{\emph{non-constant term}} circuit polynomial. For such a polynomial certainly there appears one more zero, namely the origin, and in some cases there are infinitely many zeros in addition. This occurs for instance if every outer term of the polynomial contains the variable $x_i$. 
Then, the additional zeros are zeros on the coordinate hyperplanes. 
 To be more specific, we consider Laurent polynomials. Then we can write a non-constant term circuit polynomial $f$ as an irreducible product
\[
f\;= \;\mathbf{x}^{-\boldsymbol{\alp}(j)} \cdot f^{c}, 
\]
with some exponent $\boldsymbol{\alp}(j)$ and $f^{c}$ being a constant term circuit polynomial. Hence, the original upper bound is valid on the real locus of the one nontrivial irreducible component, if we omit the obvious redundant part of the circuit polynomial. Thus, we often reduce ourselves to observing the zeros in $(\R^*)^n$, if we are only interested in a finite zero set. For convenience we define
\begin{align*}
\struc{\cV^*(f)}\; \coloneqq\; \cV(f) \cap (\R^{*})^{n}.
\end{align*}

Another instance of a non-constant term circuit polynomial possibly having infinitely many zeros is, when its Newton polytope is not full-dimensional. Consider the following simple example of an irreducible extremal circuit polynomial
\begin{align*}
r(x_1,x_2)=x_1^4+x_2^2-2x_1^2x_2=(x_1^2-x_2)^2 \in \partial C_{2,4},
\end{align*}
which has zeros at all points of the form $(v,v^2), v\in \R$. The true issue here is, that $r$ is `essentially' a polynomial in one variable, not two. See its one-dimensional Newton polytope in Figure~\ref{Fig:DegenerateCircuit}. We call a circuit polynomial \struc{\emph{degenerate}} if it has infinitely many zeros and has  not a full-dimensional Newton polytope. To  avoid the aforementioned issue, we mostly consider nondegenerate circuit polynomials.

\begin{figure}[h]
	\centering
	\begin{tikzpicture}[scale=0.87]
	\draw[very thin,color=gray] (0,0) grid (5,3);    
	\draw[->] (-0.2,0) -- (5.3,0) node[right]{}; 
	\draw[->] (0,-0.2) -- (0,3.3) node[above]{}; 
	
	\coordinate (B) at (4,0);
	\coordinate (C) at (0,2);
	
	\draw[thick] (C) -- (B);
	
	
	\node[circle, draw=red, fill=red, scale=0.7] (B) at (4,0) {};
	\node[circle, draw=red, fill=red, scale=0.7] (C) at (0,2) {};
	
	\node[draw, fill=blue, scale=0.5](F) at (2,1) {};
	
	\end{tikzpicture}
	\caption{The Newton polytope of $r(x_1,x_2)$.}
	\label{Fig:DegenerateCircuit}		
\end{figure}
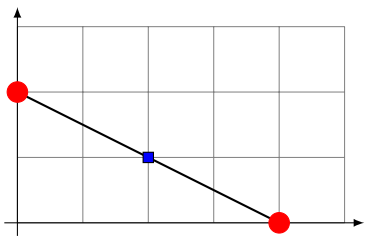
\medskip

Now we state slightly modified versions of the upper bound of affine real zeros for a nonnegative circuit polynomial $f$ due to the preceding considerations and the boundary condition, which gives a necessary and sufficient condition for $f$ to lie on the boundary of $P_{n,2d}$, see \cite[Corollaries 3.9 and 3.11]{Iliman:deWolff:Circuits}.
\begin{cor}
	A nondegenerate nonnegative circuit polynomial $f \in \partial P_{n,2d}$ has at most $2^n$ affine real zeros $\mathbf{v}$ in $(\R^*)^n$ all of which only differ in the signs of their entries. Therefore, every entry of $\mathbf{v}$ has the same norm, i.e., the zeros are of the form $(\pm v_1,\pm v_2,\ldots,\pm v_n)$.
	\label{Cor:NumberAffineZeroesModified}	
\end{cor}

\begin{cor}
	\label{Cor:CP located on boundary PSD - PROPER}
A proper circuit polynomial $f$ is located on the boundary of the cone of nonnegative polynomials, i.e., $f\in \partial P_{n,2d}$, if and only if $f_{\boldsymbol{\beta}}\in \{\pm \Theta_f\}$ and $\boldsymbol{\beta} \notin (2\N)^n$ or $f_{\boldsymbol{\beta}}=-\Theta_f$ and $\boldsymbol{\beta} \in (2\N)^n$. 	
\end{cor}

In fact, by combinatorial arguments, we can refine the observation of an upper bound and state the exact number of zeros in dependence of the inner exponent. Recall that we always count distinct zeros. 

\begin{thm}
	The number of affine real zeros $\mathbf{v}$  in $(\R^*)^n$ of a proper  nondegenerate nonnegative circuit polynomial $f \in \partial C_{n,2d}$ is $2^n$ if $\boldsymbol{\beta} \in (2\N)^n$ and $2^{n-1}$ if $\boldsymbol{\beta} \notin (2\N)^n$.
	\label{Thm:ExactNumberAffineZeroes}
\end{thm}

\begin{proof}
First note that for a single circuit polynomial $f$ we have $f\in \partial C_{n,2d}$ if and only if $f\in \partial P_{n,2d}$. 
By symmetry the number of zeros only depends on the inner exponent and every nonnegative circuit polynomial has at most one zero on every orthant, see also Corollary~\ref{Cor:CP located on boundary PSD - PROPER}. Thus, for $f\in \partial C_{n,2d}$ with $\boldsymbol{\beta}$ even we clearly have $|\cV^*(f)|=2^n$. If $\boldsymbol{\beta} \notin (2\N)^n$ we may assume without loss of generality $f_{\boldsymbol{\beta}}= - \Theta_f$. Hence a necessary condition for $f$ to have zeros is $\sgn(\mathbf{x}^{\boldsymbol{\beta}})=1$. Therefore only an even number of entries $\beta_i$ may be odd. Let $s$, $1\leq s \leq n$, be the number of odd entries in $\boldsymbol{\beta}$, then the number of affine real zeros is given by the subsequent basic calculation:
{
	\setlength{\belowdisplayskip}{-0.3cm}
\begin{align*}
2^{n-s} \cdot \sum_{k=0 \atop k \text{ even}}^s \binom{s}{k}\;=\; 2^{n-s} \cdot 2^{s-1} = 2^{n-1}.
\end{align*}
}
\end{proof}	

This result yields the following number of zeros for SONC polynomials.  

\begin{cor}
	\label{Cor:ExactNumberAffineZeroes_SONC}
	Let $p \in \partial C_{n,2d}\cap \partial P_{n,2d}$ and $p=\sum_{i=1}^{k} f_i$, where $f_i$ are proper nondegenerate nonnegative circuit polynomials for all $i$ with corresponding inner exponent $\boldsymbol{\beta}^{(i)}$. Then $|\cV^*(p)|=2^n$ if $\boldsymbol{\beta}^{(i)} \in (2\N)^n$ for all $i=1,\ldots,k$ and $1 \leq |\cV^*(p)|\leq 2^{n-1}$, with $|\cV^*(p)|\bigm| 2^{n-1}$ otherwise. In particular, if every $j$-th entry, $j=1,\ldots,n$, of each $\boldsymbol{\beta}^{(i)}$ coincides in whether $\beta_j^{(i)}$ is even or odd, then $|\cV^*(p)|= 2^{n-1}$. 
\end{cor}

\begin{proof}
	Since a SONC polynomial $p$ is a sum of nonnegative circuit polynomials $f_i$, $p$ is zero if and only if every summand $f_i$ is zero. Thus, the assertion follows by counting the common zeros of the $f_i$. Due to the special structure of the zeros of a circuit polynomial two nonnegative circuit polynomials $f_i$ and $f_l$, $i\neq l$, having more than one zero in common, have an even number of zeros in common. 
	If $f_i$ and $f_l$ have one zero in common and both $\boldsymbol{\beta}^{(i)}$ and $\boldsymbol{\beta}^{(l)}$ are even, then they have all their $2^n$ zeros in common.
	 Analogously, if both circuit polynomials have an odd inner exponent where each entry $\beta_j^{(i)}$ and $\beta_j^{(l)}$ coincide whether it is odd or even, then the zero set of $f_i$ and $f_l$ is identical, i.e., both polynomials have all their $2^{n-1}$ zeros in common.
\end{proof}	
\smallskip
The subsequent example illustrates the last case of Corollary \ref{Cor:ExactNumberAffineZeroes_SONC}.

\begin{example}
	Consider the bivariate SONC polynomial $p=f_1+f_2$ with the following two nonnegative circuit polynomials
	\begin{align*}
	f_1 \;&=\; 3/8 + 3/8\cdot x_1^4+ 1/4 \cdot x_1^2x_2^4 - x_1^2x_2 ,\\
	f_2 \;&=\; 1/8 + 1/2 \cdot x_1^4 + 3/8\cdot  x_2^8 - x_1^2x_2^3.
	\end{align*}
	Obviously, both $\beta_1^{(1)}=2$ and $\beta_1^{(2)}=2$ are even and both $\beta_2^{(1)}=1$ and $\beta_2^{(2)}=3$ are odd. Therefore, we have $\cV(f_1)=\cV(f_2)=\cV(p)=\{(1,1),(-1,1)\}$. Hence the number of zeros is $|\cV(p)|=2^{2-1}=2$. 
	\endexa
\end{example}

If we take also infinitely many zeros and sums of monomial squares into account we get the following result, which is a direct conclusion of the observations at the beginning of this section and  Corollary~\ref{Cor:ExactNumberAffineZeroes_SONC}.
\begin{cor}
	\label{Cor:InfiniteExactNumberAffineZeroes_SONC}
	Let $p \in \partial C_{n,2d}\cap \partial P_{n,2d}$ be given as in Corollary $\ref{Cor:ExactNumberAffineZeroes_SONC}$. 
	Generally, it is possible to have $|\cV(p)|=\infty$. If $\cV(p)$ is finite, then $|\cV(p)|=2^n$ or $|\cV(p)| = 2^n+1$ if $\boldsymbol{\beta}^{(i)} \in (2\N)^n$ for all $i=1,\ldots,k$, and otherwise $1\leq |\cV(p)|\leq 2^{n-1}+1$.
\end{cor}

\medskip
At this point, we insert a brief discussion on the determination of the zeros. \\
In \cite{Iliman:deWolff:Circuits} it is shown that the zeros $\mathbf{v}\in \R^n$ of the nonnegative circuit polynomial $f\in C_{n,2d}$ with $\boldsymbol{\alp}(0)=0$ and $f_0=\lam_0$, i.e., $f=\lam_0 + \sum_{j=1}^{n} f_j \mathbf{x}^{\boldsymbol{\alpha}(j)} - \Theta_f \mathbf{x}^{\boldsymbol{\beta}}$, satisfy \linebreak $|v_j|=e^{s^*_j}$ for all $j=1,\ldots,n$, where $\mathbf{s}^* \in \R^n$ is the unique vector satisfying $e^{\langle \mathbf{s}^*, \boldsymbol{\alpha}(j) \rangle}=\frac{\lam_j}{f_j}$ for all $j$. Thus, $\mathbf{s}^*$ is given by a linear system of equations.\\
In the specific case that $\New(f)=\Delta_{n,2d}$, it is even possible to exactly specify the zeros of $f$. To be more precise, consider the nonnegative circuit polynomial $f\in C_{n,2d}$ with $f=\lam_0 + \sum_{j=1}^{n} f_j x_j^{2d} - \Theta_f \mathbf{x}^{\boldsymbol{\beta}}$. Then every entry $v_j$ of every zero $\mathbf{v}\in \R^n$ of this polynomial is given by $|v_j|=(\lam_j / f_j)^{1/(2d)}$, see also \cite[Lemma 4.1]{Dressler:Iliman:deWolff:REP}. \\
An interesting question is, whether it is possible to prescribe the zeros of nonnegative circuit polynomials in the general case.\\

We now analyze the numbers of zeros in the homogeneous case. Recall that depending on the polynomial $p$, the zero set $\cV(\ovl p)$ may have no, finitely many, or infinitely many additional zeros at infinity. Therefore, we now address ourselves to the task of determining the number of real zeros additionally appearing due to homogenizing SONC polynomials. 

In the affine case, we were mainly interested in finitely many zeros, which corresponds to investigating constant term circuit polynomials. In what follows, we distinguish between SONC forms arising from homogenizing a constant term SONC polynomial and those arising from homogenizing a non-constant term SONC polynomial. 

To begin with, we consider nonnegative circuit polynomials and their homogenizations. The following result on the additional zeros in the homogeneous case will be explained in great detail and clearly described including specific representatives in which the specific number of zeros arises. Note that by properness, see Lemma~\ref{Lem:RealizabilityCP}, not all possibilities of numbers of zeros are realizable for every degree, see also the comments in the proof below. Let $\struc{\ovl C_{n+1,2d}}$ denote the cone of SONC forms in $\R[x_0,\mathbf{x}]_{n+1,2d}$.

\begin{thm}
	Let $f \in \partial C_{n,2d}$ be a nonnegative circuit polynomial and $\ovl f \in \partial \ovl C_{n+1,2d}$ be its homogenization.
	\begin{enumerate}
		\item[$(1)$] For $n=2$, we have:
		\renewcommand{\labelenumii}{\textnormal{(\roman{enumii})}}
		\begin{enumerate}
			\item If $f$ has a constant term, then $|\cV(f)|=4$ and $|\cV(\ovl f)|\in \{4,5,6\}$  if $\boldsymbol{\beta}$ is even and $|\cV(f)|=2$ and $|\cV(\ovl f)| \in \{2,3,4\}$ if $\boldsymbol{\beta}$ is odd. 
			\item If $f$ does not have a constant term, then $|\cV(f)|=5$ or infinity and $|\cV(\ovl f)|\in \{5,6,7\}$  or infinity  if $\boldsymbol{\beta}$ is even and $|\cV(f)|=3$ or infinity and $|\cV(\ovl f)| \in \{3,4,5\}$  or infinity if $\boldsymbol{\beta}$ is odd.
		\end{enumerate}
		Consequently, if $\cV(f)$ is finite, then in both cases $\ovl f$ has at most $2$ zeros at infinity in addition. 		
		\item[$(2)$]  In the general case $n>2$ by homogenizing a constant term resp. a non-constant term circuit polynomial $f\in \partial C_{n,2d}$, $\ovl f$ has either up to $3$ resp. $2$ additional zeros at infinity, or else infinitely many. Again all intermediate cases may occur.
	\end{enumerate}
	\label{Thm:ExactNumberHomogZeroes}	
\end{thm}
\vspace*{-0.2cm}
For the general case, we stress the fact that in contrast to the bivariate case, infinitely many additional zeros are also possible in the constant term case. That means if $\cV(f)$ is finite, $\ovl f$ may have infinitely many zeros. 

\begin{proof}
	Without loss of generality we assume $\New(f)$ to be $n$-dimensional for $f\in \partial C_{n,2d}$. 
	\begin{enumerate}
		\item[(1)] Independent of whether $f$ has a constant term or not, it is evident that if $\cV(f)$ is finite there are not more than two zeros at infinity, $[0:1:0]$ and $[0:0:1]$.
		\begin{enumerate}
			\item[(i)] The statements for the affine case follow directly from Theorem~\ref{Thm:ExactNumberAffineZeroes} and the assumption that $f$ has a constant term.
			 For the projective case first consider $f\in \partial C_{2,2d}$ with $\boldsymbol{\beta}$ even and its homogenization:
			\begin{align*}
			f(x_1,x_2) &=  f_0 + f_{\boldsymbol{\alp}(1)} x_1^{\alp_1(1)}x_2^{\alp_2(1)}+ f_{\boldsymbol{\alp}(2)} x_1^{\alp_1(2)}x_2^{\alp_2(2)} - \Theta_f x_1^{\beta_1}x_2^{\beta_2},\\
			\ovl f(x_0,\mathbf{x}) &= f_0 x_0^{2d}+ f_{\boldsymbol{\alp}(1)} \mathbf{x}^{\boldsymbol{\alp}(1)} x_0^{2d-|\boldsymbol{\alp}(1)|}+ f_{\boldsymbol{\alp}(2)} \mathbf{x}^{\boldsymbol{\alp}(2)} x_0^{2d-|\boldsymbol{\alp}(2)|} - \Theta_{\ovl f}\mathbf{x}^{\boldsymbol{\beta}} x_0^{2d-|\boldsymbol{\beta}|} \;.\hspace{-2cm}
			\end{align*}

			Since $\boldsymbol{\beta} \in \Int(\New(f))$, both $\beta_1$ and $\beta_2$ are non-zero and $2d-|\boldsymbol{\beta}|\neq 0$. Clearly, $\ovl f$ has at least $4$ zeros corresponding to the zeros of the affine circuit polynomial $f$. Thus, $4\leq |\cV(\ovl f)| \leq 6$.  The number of additional zeros of $\ovl f$ depends on the exponents of the monomials of $f$ corresponding to the vertices of $\New(f)$, i.e., the outer exponents of $f$. In what follows, we analyze this exponent structure in detail and point out which cases may occur. Observe that the origin is always a vertex of $\New(f)$ corresponding to the constant term. Consequently, it suffices to study the other two vertices $\boldsymbol{\alp}(1)$ and $\boldsymbol{\alp}(2)$ of $\New(f)$. Also note that at least one of these vertices has to be of full degree, i.e., $|\boldsymbol{\alp}(i)|=2d$ for at least one $i\in \{1,2\}$. There are exactly three cases. \linebreak
			Case 1: $|\cV(\ovl f)|=4$, i.e., $\ovl f$ has no additional zeros. This occurs only if one vertex of $\New(f)$ has full degree in $x_1$ and the other full degree in $x_2$. Without loss of generality we assume $\alp_1(1)=2d$ and $\alp_2(2)=2d$. Then:
			\begin{align*}
			f &=  f_0 + f_{\boldsymbol{\alp}(1)} x_1^{2d}+ f_{\boldsymbol{\alp}(2)} x_2^{2d} - \Theta_f x_1^{\beta_1}x_2^{\beta_2},\\
			\ovl f &= f_0 x_0^{2d}+ f_{\boldsymbol{\alp}(1)} x_1^{2d}+ f_{\boldsymbol{\alp}(2)} x_2^{2d} - \Theta_{\ovl f} x_1^{\beta_1}x_2^{\beta_2} x_0^{2d-|\boldsymbol{\beta}|} \;.
			\end{align*}
			Obviously, $\ovl f$ only has the zeros corresponding to those of $f$.\\
			Case 2: $|\cV(\ovl f)|=5$. This happens, when the outer exponents of $f$ are of one of the following three structures:
		One of the vertices is of full degree in exactly one variable, we may assume $\alp_1(1)=2d$, and for the second vertex 
			\begin{enumerate}
				\item[(a)] only the other variable $x_2$ occurs, but not of full degree, i.e., $\alp_1(2)=0$ and $\alp_2(2)<2d$, 
				\item[(b)] both variables occur but the vertex is not of full degree, hence $\alp_1(2)\neq 0$,  $\alp_2(2)\neq 0$, and $|\boldsymbol{\alp}(2)|<2d$, or
				\item[(c)] both variables occur and the vertex is of full degree, thus $\alp_1(2)\neq 0$,  $\alp_2(2)\neq 0$, and $|\boldsymbol{\alp}(2)|=2d$.
			\end{enumerate}
			For Case (a) observe that, with $\alp_2(2)<2d$,
			\begin{align*}
			f &=  f_0 + f_{\boldsymbol{\alp}(1)} x_1^{2d}+ f_{\boldsymbol{\alp}(2)} x_2^{\alp_2(2)} - \Theta_f x_1^{\beta_1}x_2^{\beta_2},\\
			\ovl f &= f_0 x_0^{2d}+ f_{\boldsymbol{\alp}(1)} x_1^{2d}+ f_{\boldsymbol{\alp}(2)} x_2^{\alp_2(2)}x_0^{2d-\alp_2(2)} - \Theta_{\ovl f} x_1^{\beta_1}x_2^{\beta_2} x_0^{2d-|\boldsymbol{\beta}|} \;.
			\end{align*}
			The additional zero is $[0:0:1]$. The other cases follow analogously.\\
			Case 3: $|\cV(\ovl f)|=6$. This case arises if one of the vertices is full-dimensional and of full degree, without loss of generality $|\boldsymbol{\alp}(1)|=2d$, and the other vertex exhibits one of the following structures:
			\begin{enumerate}
				\item[(a)] only one variable $x_1$ or $x_2$ occurs, but not of full degree, hence for \linebreak example $\alp_1(2)=0$ and $\alp_2(2)<2d$, 
				\item[(b)] both variables occur but the vertex is not of full degree, thus $\alp_1(2)\neq 0$,  $\alp_2(2)\neq 0$, and $|\boldsymbol{\alp}(2)|<2d$, or
				\item[(c)] both variables occur and this vertex is also of full degree, i.e., $\alp_1(2)\neq 0$,  $\alp_2(2)\neq 0$, and $|\boldsymbol{\alp}(2)|=2d$.
			\end{enumerate}
			Similarly as in Case 2, we consider only Case (c), the others follow analogously. For (c) we have:
			\begin{align*}
			f &=  f_0 + f_{\boldsymbol{\alp}(1)} \mathbf{x}^{\boldsymbol{\alp}(1)}+ f_{\boldsymbol{\alp}(2)}  \mathbf{x}^{\boldsymbol{\alp}(2)} - \Theta_f  \mathbf{x}^{\boldsymbol{\beta}},\\
			\ovl f &= f_0 x_0^{2d}+ f_{\boldsymbol{\alp}(1)} \mathbf{x}^{\boldsymbol{\alp}(1)} + f_{\boldsymbol{\alp}(2)}  \mathbf{x}^{\boldsymbol{\alp}(2)} - \Theta_{\ovl f}  \mathbf{x}^{\boldsymbol{\beta}} x_0^{2d-|\boldsymbol{\beta}|} \;.
			\end{align*}
			Obviously, the additional zeros are $[0:0:1]$ and $[0:1:0]$. \\
			To finish the proof of (1)(i), it remains to consider $f\in \partial C_{2,2d}$ with $\boldsymbol{\beta} \notin (2\N)^2$. However, here the reasoning is as for $\boldsymbol{\beta}$ even except that the number of zeros starts with $|\cV(\ovl f)|=2$. We point out one small difference for degree $4$. Despite the general existence of a proper bivariate circuit polynomial of degree $4$ with an odd inner exponent, there does not occur one with the exponent structure of Case 3, i.e., for two additional zeros. Because such a circuit polynomial would be only a sum of monomial squares since there is no inner term. Whereas there is no proper bivariate circuit polynomial with even inner exponent for degree 4. Hence, by homogenizing $f\in C_{2,4}$, we get at most one zero at infinity.
			\item[(ii)] This statement follows immediately from (i) and the considerations at the beginning of Section~\ref{Sec:RealZeros}.
		\end{enumerate}
		%
			%
		\vspace*{2mm}
		\item[(2)] First of all we show that even a homogenized constant term nonnegative circuit polynomial $\ovl f \in \partial \ovl C_{n+1,2d}$ with $n\geq 3$  possibly has infinitely many zeros at infinity. For some $\ovl f \in \partial \ovl C_{n+1,2d}$ the additional zeros are of the form $[0:0:v_2:v_3:\cdots:v_n]$, where,  without loss of generality, we assumed $v_1=0$. The existence of such a case can easily be seen by the following examples. Let $f\in \partial C_{n,2d}$ be either of the form
		\begin{align*}
		f =  f_0 + f_{\boldsymbol{\alp}(1)}x_1^{2d}+ f_{\boldsymbol{\alp}(2)}  x_1^{\alp_1(2)}x_2^{2d-\alp_1(2)}+\cdots+f_{\boldsymbol{\alp}(n)}  x_1^{\alp_1(n)}x_n^{2d-\alp_1(n)} - \Theta_f  \mathbf{x}^{\boldsymbol{\beta}} , \hspace{-1cm}
		\end{align*}
		with $\alp_1(i) \in 2\N^*$ and $\alp_1(i)<2d$ \;for\; $i= 2,\ldots,n$. Or	
		\begin{align*}
		f =  f_0 + f_{\boldsymbol{\alp}(1)}x_1^{2d}+ f_{\boldsymbol{\alp}(2)}  x_2^{<2d}+\cdots+f_{\boldsymbol{\alp}(n)}x_n^{<2d} - \Theta_f  \mathbf{x}^{\boldsymbol{\beta}},
		\end{align*}
		 where \struc{$x_i^{<2d}$} abbreviates $x_i^{\alp_i(i)}$ with $\alp_i(i)<2d$.
		Obviously, in both examples the zeros of $\ovl f$ are $[0:0:v_2:\cdots:v_n]$, with $(v_2,\ldots,v_n) \in \R^{n-1}\setminus \{\boldsymbol{0}\}$. Thus $|\cV(\ovl f)|=\infty$. \\
		We now study again the additional zeros of $\ovl f\in \partial \ovl C_{n+1,2d}$. In what follows, we only discuss the constant term case. But we show that by homogenizing a non-constant term circuit polynomial $f$, the form $\ovl f$ cannot have three zeros at infinity. We proceed as in (1) by analyzing the vertex constellations of $\New(f)$, where, by assumption, we exclude the origin from consideration. Here, there are four cases.\\
		Case 1: $|\cV(\ovl f)|=|\cV(f)|$. Like for $n=2$ the only way where no zeros at infinity appear is if all vertices of $\New(f)$, $f\in C_{n,2d}$, are of full degree in one variable each. \linebreak
		Case 2: $|\cV(\ovl f)|=|\cV(f)|+1$. This happens if all but one vertex  of $\New(f)$ is of full degree in one different variable and the last vertex consists either of the not yet appearing variable but not of full degree or of various variables optional if of full degree or not. For instance, consider $f\in \partial C_{n,2d}$ with
		\begin{align*}
		f =  f_0 + f_{\boldsymbol{\alp}(1)} x_1^{2d}+ f_{\boldsymbol{\alp}(2)} x_2^{2d}+\cdots+ f_{\boldsymbol{\alp}(n)} x_n^{<2d} - \Theta_f \mathbf{x}^{\boldsymbol{\beta}}. 
		\end{align*}
		By homogenizing $f$, we get the additional zero $[0:0:\cdots:0:v_n]$, $v_n \in \R^*$.\\
		Case 3: $|\cV(\ovl f)|=|\cV(f)|+2$. Here, $n-1$ of the non-origin vertices of $\New(f)$ have to be of full degree in one variable, whereas for the remaining two vertices different exponent structures are possible. For example, consider
		\begin{align*}
		f =  f_0 + f_{\boldsymbol{\alp}(1)} x_1^{2d}+ \cdots +f_{\boldsymbol{\alp}(n-2)} x_{n-2}^{2d}+ f_{\boldsymbol{\alp}(n-1)} x_{n-1}^{<2d} +f_{\boldsymbol{\alp}(n)} x_{n-1}^{\alp} x_n^{2d-\alp} - \Theta_f \mathbf{x}^{\boldsymbol{\beta}}, \hspace{-1cm}
		\end{align*}
		where $\alp \in 2\N^*$ and $\alp<2d$.
		The homogenization $\ovl f$ has the two supplementary zeros $[0:0:\cdots:0:v_{n-1}:0]$ and $[0:0:\cdots:0:v_n]$ with $v_{n-1},v_n \in \R^*$.\\
		Case 4: $|\cV(\ovl f)|=|\cV(f)|+ 3$. This case arises if three vertices of $\New(f)$ are of full degree in two variables, which use three variables pairwise. All other vertices are of full degree in one different variable each:
		\begin{gather*}
		f =  f_0 + f_{\boldsymbol{\alp}(1)} x_1^{2d}+ \cdots +f_{\boldsymbol{\alp}(n-3)} x_{n-3}^{2d}  \\
		+ f_{\boldsymbol{\alp}(n-2)}x_{n-2}^{\alp}x_{n-1}^{2d-\alp} + f_{\boldsymbol{\alp}(n-1)} x_{n-1}^{\delta}x_n^{2d-\delta} + f_{\boldsymbol{\alp}(n)} x_n^{\gamma} x_{n-2}^{2d-\gamma}  - \Theta_f \mathbf{x}^{\boldsymbol{\beta}}, 
		\end{gather*}
		where $\alp, \delta,\gamma \in 2\N^*$ and each is strictly smaller than $2d$.
		The three additional zeros of $\ovl f$ are $[0:0:\cdots:0:v_{n-2}:0:0]$, $[0:0:\cdots:0:v_{n-1}:0]$, and $[0:0:\cdots:0:v_n]$, with $v_{n-2},v_{n-1},v_n \in \R^*$. Obviously, this case cannot occur if $f$ does not have a constant term, since then $f$ already has infinitely many (affine) zeros.\\   
		It remains to show that for finite $\cV(\ovl f)$, the number of additional zeros is bounded by $3$. We proceed by contradiction to prove that we may exclude the case of $4$ additional zeros, all other  follow analogously. 
		Note that in this case necessarily $n\geq 4$.\linebreak Thus, to show the $n$-variate case, it suffices to exclude that $\ovl f$ gets $4$ additional zeros for  $f\in C_{4,2d}$. Suppose $f\in C_{4,2d}$ with $|\cV(\ovl f)|=|\cV(f)|+4$. The additional zeros are $[0:1:0:0:0]$, $[0:0:1:0:0]$, $[0:0:0:1:0]$, and $[0:0:0:0:1]$. In order for such a zero set to exist, one part of $f$ has to be a sum of monomials, where every term consists of a product of two variables each $x_ix_j$, $ i,j=\{1,\ldots,4\},i\neq j,$ and all pairings have to appear. Each summand has to be of full degree. To receive a circuit polynomial, the mentioned terms of variable pairings together with the origin form the outer terms of $f$ and $x_1^{\beta_1}\cdots x_4^{\beta_4}$ forms the inner term. Hence, there are $\binom{4}{2}+1=7$ outer terms, which is a contradiction to $f$ being a circuit polynomial, because $f$ may only have up to $5$ outer terms. 
	\end{enumerate}
	\vspace{-\topsep}\vspace{-\topsep}
\end{proof}	
\vspace*{1mm}
\begin{remark}
	The considerations of the final step in the last proof are in line with the possibility to receive $3$ additional zeros in the constant term case. Following the train of thought starting with $f\in C_{3,2d}$, we get $\binom{3}{2}+1=4$ outer terms. This is consonant to the number of vertices of a simplex.
\end{remark}
\vspace*{1mm}
A direct corollary for SONC forms can be drawn from the above arguments.
\enlargethispage{\baselineskip}
\begin{cor}
	Let $p \in \partial C_{n,2d}\cap \partial P_{n,2d}$ be a SONC polynomial with $p=\sum_{i=1}^{k} f_i$, where $f_i$ are proper nonnegative circuit polynomials for all $i$ with corresponding inner exponent $\boldsymbol{\beta}^{(i)}$. Consider the homogenization $\ovl p \in \partial \ovl C_{n+1,2d}\cap \partial \ovl P_{n+1,2d}$.
	\begin{enumerate}
		\item[$(1)$] For $n+1=3:$
		\renewcommand{\labelenumii}{\textnormal{(\roman{enumii})}}
		\begin{enumerate}
			\item If $p$ has a constant term, then $|\cV(p)| \leq |\cV(\ovl p)|\leq |\cV(p)|+2$. More precisely,  if $\boldsymbol{\beta}^{(i)} \in (2\N)^2$ for all $i=1,\ldots,k$, then 
			$4 \leq |\cV(\ovl p)|\leq 6$, if every $j$-th entry of each $\boldsymbol{\beta}^{(i)}$ coincides in whether $\beta_j^{(i)}$ is even or odd, then $2 \leq |\cV(\ovl p)|\leq 4$, and otherwise $1 \leq |\cV(\ovl p)|\leq 4$. In all three cases the given bounds are sharp and the intermediate cases occur, with exception of $2d=4$, where the upper bound for the second and the third case is $3$. 
			\item If $p$ does not have a constant term, then each $f_i$ is a non-constant term circuit polynomial and either $|\cV(p)|\leq |\cV(\ovl p)|\leq |\cV(p)|+2$ or $|\cV(\ovl p)|= \infty$.\linebreak More precisely, if $\cV(p)$ is finite then we have three cases: If $\boldsymbol{\beta}^{(i)} \in (2\N)^2$ for all $i=1,\ldots,k$, then $5 \leq |\cV(\ovl p)|\leq 7$,
			if every $j$-th entry of each $\boldsymbol{\beta}^{(i)}$ coincides in whether $\beta_j^{(i)}$ is even or odd, then $3 \leq |\cV(\ovl p)|\leq 5$, and  otherwise $1 \leq |\cV(\ovl p)|\leq 5$.
		Again, the bounds are sharp and the intermediate cases occur. The only exception is for degree $4$, where the number of zeros is either $2$ or $3$.
		\end{enumerate}	
		\item[$(2)$] If $n+1\geq 4$ either
		\renewcommand{\labelenumii}{\textnormal{(\alph{enumii})}}
		\begin{enumerate}
			\item $|\cV(p)|\leq |\cV(\ovl p)|\leq |\cV(p)|+3$. In particular: 
			$2^n \leq |\cV(\ovl p)|\leq 2^n+3$ if  $\boldsymbol{\beta}^{(i)} \in (2\N)^{n}$ for all $i=1,\ldots,k$ and $2^{n-1} \leq |\cV(\ovl p)|\leq 2^{n-1}+3$ if every $j$-th entry of each $\boldsymbol{\beta}^{(i)}$ coincides in whether $\beta_j^{(i)}$ is even or odd, or 
			\item $|\cV(\ovl p)|= \infty$.
		\end{enumerate}
		The bounds of $(a)$ are sharp as well and all intermediate cases can occur.
	\end{enumerate}
	\label{Cor:NumbersZerosHomoSONC}
\end{cor}

\begin{remark}
	In particular, we point out that in contrast to a SONC polynomial with a constant term its homogenization may have infinitely many zeros (at infinity) even in the case $n+1=3$. Furthermore, the properness condition in Corollary~\ref{Cor:NumbersZerosHomoSONC} is necessary since we also take lower bounds on the zero set into account. 
\end{remark}

\smallskip
Before we proceed to analyze consequences of the results on real zeros of SONC polynomials, we provide some explicit examples demonstrating the considered cases in the proof of Theorem~\ref{Thm:ExactNumberHomogZeroes}.
\vspace*{0.15cm}
\begin{example}
	~\vspace*{-0.15cm}	\begin{enumerate}
		\item[(i)] First we give an example for $|\cV(\ovl f)|=|\cV(f)|$. Let $f\in C_{2,4}$ be the following nonnegative circuit polynomial
		\[
		f=\frac{1}{2} + x_1^4 + x_2^4 - 2x_1x_2.
		\]
		The zeros of $f$ are $\mathbf{v}_1=\lf(\frac{1}{\sqrt{2}},\frac{1}{\sqrt{2}}\ri)$ and $\mathbf{v}_2=\lf(-\frac{1}{\sqrt{2}},-\frac{1}{\sqrt{2}}\ri)$. Homogenizing $f$ yields $\ovl f= \frac{1}{2}x_0^4 + x_1^4 + x_2^4 - 2x_1x_2x_0^2$ and $\cV(\ovl f)=\lf\{\lf[1:\frac{1}{\sqrt{2}}:\frac{1}{\sqrt{2}}\ri],\lf[1:-\frac{1}{\sqrt{2}}:-\frac{1}{\sqrt{2}}\ri]\ri\}$.
		\item[(ii)] For $|\cV(\ovl f)|=|\cV(f)|+1$ consider 
		\[
		f= \frac{1}{3}+\frac{1}{6}x_1^6+\frac{1}{2}x_1^2x_2^4 - x_1^2x_2^2,
		\]
		its $4$ zeros are $(\pm 1,\pm 1)$. Then $\ovl f= \frac{1}{3}x_0^6+\frac{1}{6}x_1^6+\frac{1}{2}x_1^2x_2^4 - x_1^2x_2^2x_0^2$, which has the additional zero at infinity $[0:0:1]$.
		\item[(iii)] The Motzkin polynomial provides an example for the case $|\cV(\ovl f)|=|\cV(f)|+2$: 
		\[
		f=1+x_1^4x_2^2+x_1^2x_2^4-3x_1^2x_2^2,
		\]
		with zeros $(\pm 1,\pm 1)$. The Motzkin form $\ovl f= x_0^6+x_1^4x_2^2+x_1^2x_2^4-3x_1^2x_2^2x_0^2$ additionally has the zeros  $[0:1:0]$ and $[0:0:1]$.\\
		An example of a non-constant term circuit polynomial for this instance is the Choi-Lam polynomial
		\[
		S=x_1^4x_2^2+x_2^4+x_2^2-3x_1^2x_2^2,
		\]
		with zero set $\cV(S)=\{(1,1),(1,-1),(-1,1),(-1,-1),(0,0)\}$, see  \cite{Choi:Lam:OldHilbertQuestion,Choi:Lam:ExtremalpsdForms}. Its homogenization is $\ovl S= x_1^4x_2^2+x_0^2x_2^4+x_0^4x_2^2-3x_0^2x_1^2x_2^2$ and has $[0:0:1],[0:1:0]$ as additional zeros.
		\item[(iv)] The subsequent polynomial $f\in C_{3,8}$ serves as an example for $|\cV(\ovl f)|=|\cV(f)|+3$:
		\[
		f=5+x_1^4x_2^4+x_2^4x_3^4+x_1^4x_3^4-8x_1x_2x_3.
		\]
		Here $\cV(f)=\lf\{(1,1,1),(1,-1,-1),(-1,1,-1),(-1,-1,1) \ri\}$. The homogenization $\ovl f= 5x_0^8+x_1^4x_2^4+x_2^4x_3^4+x_1^4x_3^4-8x_1x_2x_3x_0^5$ additionally has the following zeros $[0:1:0:0]$, $[0:0:1:0]$, and $[0:0:0:1]$.
		\item[(v)] Finally, we provide an example for $|\cV(\ovl f)|=\infty$ if $\cV(f)$ is finite. Via homogenizing
		\[
		f= 1+x_1^4x_2^2x_3^2+x_1^2x_2^4x_3^2+x_1^2x_2^2x_3^4-4x_1^2x_2^2x_3^2,
		\]
		the form $\ovl f$ has the following additional zeros $[0:0:v_2:v_3]$, $[0:v_1:0:v_3]$, and $[0:v_1:v_2:0]$ with $(v_i,v_j)\in \R^2\setminus \{\boldsymbol{0}\}, i\neq j$. Thus, $\ovl f $ has infinitely many zeros. 
	\end{enumerate}
	\endexa
\end{example}

\subsection{Consequences of the zero statements}
\label{SubSec:RealZeros_B-numbers}

In this subsection we discuss an interesting property of SONC forms resulting from the knowledge about their real zeros.

In \cite{CLR:realzeros} Choi, Lam, and Reznick considered the numbers $B_{n+1,2d}$  and $B'_{n+1,2d}$, defined as $\sup|\cV(\ovl p)|$, where $\ovl p$ ranges over all forms in $\ovl P_{n+1,2d}$ and $\ovl \Sigma_{n+1,2d}$ respectively, with $\sup|\cV(\ovl p)|<\infty$, see also Section~\ref{SubSec:Prelim_PSD-SOS}. The authors noticed that the determination of these numbers is quite challenging and presented some partial results. Moreover, they observed that for general $n$ and $d$ it is unclear if $B_{n+1,2d}$ always needs to be finite. See also Theorem~\ref{Thm:B-numbersSOS&PSD} for results in special cases.
Inspired by this, we define an analog number for SONC forms.
\begin{definition}
	\label{Def:B''}
{
		\setlength{\belowdisplayskip}{-0.35cm}
	\begin{align*}
	\struc{B''_{n+1,2d}} \coloneqq \sup_{\substack{\ovl p\in \ovl C_{n+1,2d}\\|\cV(\ovl p)| < \infty}}|\cV(\ovl p)|.
	\end{align*}
}
	\endexa	
\end{definition}	

A crucial difference to the numbers $B_{n+1,2d}$  and $B'_{n+1,2d}$ is that in our case such a number $B''_{n+1,2d}$ is \emph{always finite} and actually can be given \emph{explicitly}.

\begin{thm}
	Let $B''_{n+1,2d}$ be as in Definition~$\ref{Def:B''}$, then:
	\renewcommand{\labelenumi}{\textnormal{(\arabic{enumi})}}
	\begin{enumerate}
		\item Special case $d=1:$ $B''_{2,2}=1$. 
		\item $B''_{2,4}=2$, $B''_{2,6}=3$, and $B''_{2,2d}=4$ for $2d \geq 8$.
		\item $B''_{3,4}=3$ and $B''_{3,2d}=7$ for $2d \geq 6$.
		\item For all $n+1\geq 4:$   $B''_{n+1,2d}=2^{2d-2}+3$ for $2d<n+1$, $B''_{n+1,2d}=2^{n-1}+3$ if $n+1 \leq 2d < 2(n+1)$, and $B''_{n+1,2d}=2^{n}+3$ for $2(n+1)\leq 2d$.
	\end{enumerate}
	\label{Thm:B''}	
\end{thm}	

\begin{proof}
	\renewcommand{\labelenumi}{(\arabic{enumi})}
	~\vspace*{-0.15cm}	\begin{enumerate}
		\item For $d=1$ we have a special case, since the only possibility for a proper SONC form of degree $2$ is the circuit form $\ovl f= f_{\boldsymbol{\alp}(0)}x_0^2 + f_{\boldsymbol{\alp}(1)}x_1^2 - \Theta_f x_0x_1$, which only has one zero $[1:1]$. Even if we consider a sum of monomial squares, the only zero in the case of degree $2$ would be $[0:0] \notin \P^2$.
		\item First, note that the maximum number of zeros of a sum of monomial squares in the case $n+1=2$ is $2$. Namely the single monomial square $x_0^2x_1^2$ has the zeros $[0:1]$ and $[1:0]$. If we consider proper SONC forms, then the number of zeros depends on the degree, because certain vertex constellations are only possible from a certain degree on. For $2d=4$ we have, up to renumbering of the variables, only two possible circuit forms $\ovl f_1= f_{\boldsymbol{\alp}(0)}x_0^4+f_{\boldsymbol{\alp}(1)}x_1^4-\Theta_{f_1}x_0^2x_1^2$ with zeros $[1:1],[1:-1]$ and $\ovl f_2=f_{\boldsymbol{\alp}(0)}x_0^2x_1^2+f_{\boldsymbol{\alp}(1)}x_0^4-\Theta_{f_2}x_0^3x_1$ with zeros $[1:1], [0:1]$. Therefore, the first assertion in (2) holds. The second follows by observing that for $2d=6$ there exists for the first time a circuit form with the following term structure: one outer term consists of both variables and the inner term has an even exponent, i.e., $\ovl f= f_{\boldsymbol{\alp}(0)}x_0^2x_1^4+f_{\boldsymbol{\alp}(1)}x_0^6-\Theta_f x_0^4x_1^2$. This yields the zeros $[1:1], [1:-1],$ and $[0:1]$. Lastly, if $2d\geq8$, a circuit form with even inner exponent exists, for which both outer terms consist of both variables: $\ovl f= f_{\boldsymbol{\alp}(0)}x_0^2x_1^6+f_{\boldsymbol{\alp}(1)} x_0^6x_1^2- \Theta_f x_0^4x_1^4$. It has the four zeros $[1:1], [1:-1], [0:1]$, and $[1:0]$. Obviously, a bivariate SONC form cannot have more than $4$ zeros. 
		\item Observe that in the case of $n+1=3$ the maximum number of zeros by a sum of monomial squares is $3$. More precisely we consider without loss of generality the following sum of monomial squares in degree $4$: $\ovl m=x_0^2x_1^2+x_1^2x_2^2+x_0^2x_2^2$. Obviously, $[1:0:0],[0:1:0]$, and $[0:0:1]$ are the zeros of $\ovl m$.\\ For reasons of realizability, see Lemma~\ref{Lem:RealizabilityCP}, a proper circuit form of degree $4$ has an odd inner exponent, and for $2d\geq 6$, also a proper circuit form with even inner exponent is possible. Thus, the statements follow immediately by Corollary~\ref{Cor:NumbersZerosHomoSONC}~(1) and the preliminary consideration.  
		\item The last two assertions are a direct result of Lemma~\ref{Lem:RealizabilityCP} and Corollary~\ref{Cor:NumbersZerosHomoSONC}~(2). 
		Note that as for $n+1=3$, the maximum number of zeros by a sum of monomial squares is $3$.
		In the case $2d<n+1$ there exists no proper circuit form. Though there are SONC forms $\ovl p$ consisting of a sum of a proper circuit form $\ovl f$ and a sum of monomial squares. We know that a proper circuit form with odd inner exponent exists if the number of variables $(n+1)$ is equal to $2d$. Hence, if $2d<n+1$ we have the following SONC form:
		\begin{align*}
		\ovl p= \ovl f + x_{2d+1}^{2d} + \cdots + x_n^{2d},
		\end{align*} 
		\noindent where $\ovl f$ is a $2d$-variate proper circuit form. Thus, $|\cV(\ovl p)|=|\cV(\ovl f)|$, which leads to the equality $B''_{n+1,2d}=B''_{2d,2d}$. By Corollary~\ref{Cor:NumbersZerosHomoSONC}~(2), with $n+1=2d$, we have $B''_{2d,2d}=2^{(2d-1)-1}+3$.
	\end{enumerate}	
	\vspace{-\topsep}
\end{proof}

The following example illustrates the considerations of case (4) for $2d<n+1$ in the proof above.
\begin{example}
	We want to verify the calculation $B''_{7,4}=2^{4-2}+3=7$. Let $\ovl p \in \partial \ovl P_{7,4}$ be a SONC form. Obviously $4<7$, therefore we search a proper $4$-variate circuit form. Consider for instance  
	\[\ovl f= \frac{1}{4}x_0^2x_1^2+\frac{1}{4}x_1^2x_2^2+\frac{1}{4}x_2^2x_0^2+\frac{1}{4}x_3^4-x_0x_1x_2x_3. \] 
	The zero set of this form is 
	\begin{gather*}
	\cV(\ovl f)=\lf\{ [1:1:1:1], [1:1:-1:-1], [1:-1:1:-1], [-1:1:1:-1],\ri. \\ \lf.[1:0:0:0], [0:1:0:0], [0:0:1:0] \ri\}.
	\end{gather*}
	Hence, $|\cV(\ovl f)|=7$. Thus, the SONC form $\ovl p$, 
	\[
	\ovl p= \frac{1}{4}x_0^2x_1^2+ \frac{1}{4}x_1^2x_2^2+ \frac{1}{4}x_2^2x_0^2+ \frac{1}{4}x_3^4- x_0x_1x_2x_3 + x_4^4+x_5^4+x_6^4,
	\]
	has the same number of zeros as $\ovl f$, namely $7$. 
	\endexa
\end{example}
\smallskip
We conclude this section with a short comparison of the $``B$-numbers'' of the three different cones $\ovl P_{n+1,2d}, \ovl \Sig_{n+1,2d}$, and $\ovl C_{n+1,2d}$.
\vspace*{0.15cm}
\begin{remark}
	\renewcommand{\labelenumi}{(\roman{enumi})}
	~\vspace*{-0.15cm}	
	\begin{enumerate}
		\item First note that $B_{2,2}=B'_{2,2}=B''_{2,2}=1$, which is in line with the fact, that for $(n+1,2d)=(2,2)$ the three cones coincide, see Theorem~\ref{Thm:MissingPieceConeContainment}.
		\item In the bivariate case one has $B_{2,2d}=B'_{2,2d}=d$, which equals $B''_{2,2d}$ for $2d\leq 8$. Therefore, we have a first difference in the number of real zeros for degree $10$.
		\item In the case $n+1=3$ we have the following observations, $B_{3,4}=B'_{3,4}=B''_{3,4}=4$. But from degree $6$ on, there are differences in the numbers of zeros: $B_{3,6}=10$, $B'_{3,6}=9$, and $B''_{3,6}=7$. 
		\item Finally, consider $n+1=4$. Here, we already have differences for quartics: \linebreak $B_{4,4}=10$, $B'_{4,4}=8$, and $B''_{4,4}=7$. 
	\end{enumerate}
	
\end{remark}
\medskip

\section{Exposed Faces of the SONC Cone in Small Dimension and Dimension Bounds}
\label{Sec:ExposedFaces}

The aim of this section is to provide a first approach to the study of the exposed faces of $C_{n,2d}$. Understanding the facial structure of the cone as well as the relationship between $P_{n,2d}$ and $C_{n,2d}$ is interesting from many perspectives in both pure and applied real algebraic geometry. Unfortunately, as mentioned in the introduction, even for the cones $P_{n,2d}$ and $\Sig_{n,2d}$ this endeavor is still an active area of research, which is not yet well understood. 
Building upon the results of the real zeros of Section~\ref{Sec:RealZeros} we analyze the dimensions of the exposed faces of $C_{n,2d}$ and compare those with the exposed faces of $P_{n,2d}$. 

We begin with providing a brief theoretical overview of exposed faces, where we also recall results for the exposed faces of $P_{n,2d}$ and $\Sig_{n,2d}$. Afterward, we derive estimates for the dimensions of the exposed faces of $C_{n,2d}$ and study some first special cases in small dimension, leading to interesting directions for further research.\\

In what follows, we restrict ourselves to the affine case.\\
Recall from the introduction that given a convex set $S\subset \R^n$, a face $F$ of $S$ is exposed if there exists an affine hyperplane $H$ in $\R^n$ such that $F=S\cap H$. 
Let $\struc{\Gamma}$ be a finite set of points in $\R^{n}$. The polynomials in $C_{n,2d}$ vanishing at all points of $\Gamma$ form an \emph{exposed face} of $C_{n,2d}$, which we define as $C_{n,2d}(\Gamma)$:
\begin{align*}
\struc{C_{n,2d}(\Gamma)}\;\coloneqq\; \{p\in C_{n,2d}:\;p(\mathbf{s})=0 \text{ for all } \mathbf{s}\in \Gamma\}.
\end{align*}
Analogously, let $\struc{P_{n,2d}(\Gamma)}$ and $\struc{\Sig_{n,2d}(\Gamma)}$ denote the exposed faces  of $P_{n,2d}$ and $\Sig_{n,2d}$ respectively, i.e.,  $P_{n,2d}(\Gamma)$ and $\Sig_{n,2d}(\Gamma)$ are the sets of all polynomials in $P_{n,2d}$ and $\Sig_{n,2d}$, respectively, that vanish at all points of $\Gamma$. In fact, any exposed face of $P_{n,2d}$ has this description, see \cite{Blekherman:Parrilo:Thomas}. We start with collecting some observations for $P_{n,2d}(\Gamma)$ and $\Sig_{n,2d}(\Gamma)$. For this, following \cite{Reznick:OnHilbertsConstructionofPositivePolys}, we denote by $\struc{I(\Gamma)_{r,2d}}$ the vector space of those polynomials $p\in \R[\mathbf{x}]_{n,2d}$, which have an $r$-th order zero at each $\mathbf{s}\in \Gamma$. Then, 
\begin{align*}
\struc{I(\Gamma)_{1,d}}&\;\coloneqq\;\{p\in \R[\mathbf{x}]_{n,d}: \;p(\mathbf{s})= 0 \text{ for all } \mathbf{s}\in \Gamma\},\\
\struc{I(\Gamma)_{2,2d}}&\;\coloneqq\;\{p\in \R[\mathbf{x}]_{n,2d}: \;\nabla p(\mathbf{s})= 0 \text{ for all } \mathbf{s}\in \Gamma\}.
\end{align*}
Here $\struc{\nabla p}$ denotes the \emph{gradient} of $p$. Without degree bounds $I(\Gamma)_{1}$ is the vanishing ideal of $\Gamma$, and $I(\Gamma)_{2}$ is the second symbolic power of $I(\Gamma)_{1}$. \\
Clearly, $P_{n,2d}(\Gamma) \subset I(\Gamma)_{2,2d}$, since for nonnegative polynomials $p$ zeros are local \linebreak minima, which implies that the gradient of $p$ at the zeros must vanish as well. Whereas for the set of exposed faces of the SOS cone we have $\Sig_{n,2d}(\Gamma)\subset I(\Gamma)_{1,d}^2$, where \linebreak $\struc{I(\Gamma)_{1,d}^2}=\lf\{\sum_i \alp_if_ig_i : f_i,g_i \in I(\Gamma)_{1,d}, \alp_i \in \R\ri\}=\lf\{\sum_i \alp_ih_i^2: h_i \in I(\Gamma)_{1,d}, \alp_i \in \R\ri\}$. \!Actually, one can show that this inclusion is full-dimensional, i.e., $\dim(\Sig_{n,2d}(\Gamma))= \dim(I(\Gamma)_{1,d}^2)$. Obviously $I(\Gamma)_{1,d}^2\subseteq I(\Gamma)_{2,2d}$ holds. Thus subsequent questions concern the full-dimen\-sionality of $P_{n,2d}(\Gamma)$ in $I(\Gamma)_{2,2d}$ and then, the equality of $I(\Gamma)_{2,2d}$ and $I(\Gamma)_{1,d}^2$. These questions are discussed in \cite{Blekherman:Iliman:Kubitzke:ExposedFaces}. Therein the authors showed $\dim(P_{n,2d}(\Gamma)) = \dim(I(\Gamma)_{2,2d})$ under some assumptions on the set $\Gamma$, namely, if $\Gamma$ is ``$d$-independent''. Moreover, they provided an answer for the second question again  under some assumptions on the set $\Gamma$ and characterized those cases, where $\dim(I(\Gamma)_{2,2d})$ is strictly greater than $\dim(I(\Gamma)_{1,d}^2)$.\\
In what follows, we use the subsequent observation for the computation of $\dim(P_{n,2d}(\Gamma))$. Since in $n$ variables a second order zero imposes $n+1$ linear conditions which not \linebreak necessarily are all independent, we get
\begin{align}
\label{Eq:AlexHirsch}
\dim(I(\Gamma)_{2,2d})\;\geq\; \dim(\R[\mathbf{x}]_{n,2d}) - |\Gamma|\cdot(n+1)\;=\;\binom{n+2d}{2d}- |\Gamma|\cdot(n+1).
\end{align}
By the Alexander-Hirschowitz Theorem \cite{Miranda:AlexanderHirschowitzThm} it follows that generically, with \linebreak exception of $2d=2$, we have equality in the above inequality.  

\medskip

We now analyze $C_{n,2d}(\Gamma)$. Clearly, we have $C_{n,2d}(\Gamma)\subseteq P_{n,2d}(\Gamma)$. Immediate \linebreak subsequent questions are: What is the dimension of $C_{n,2d}(\Gamma)$? Are there cases where $C_{n,2d}(\Gamma)$ is full-dimensional in $P_{n,2d}(\Gamma)$, i.e., $\dim(C_{n,2d}(\Gamma)) = \dim(P_{n,2d}(\Gamma)) $? 

\smallskip

Henceforth, we consider $|\Gamma|=2^n$ and $|\Gamma|=2^{n-1}$.
To begin with, we state an important observation regarding the dimension of $C_{n,2d}(\Gamma)$. Obviously, $\dim(C_{n,2d}(\Gamma))$ equals the number of linear independent SONC polynomials in $C_{n,2d}$ vanishing at all points of $\Gamma$. Thus, it suffices to study nonnegative circuit polynomials or, in fact, agiforms $f$, where all entries of all zeros $\mathbf{v}$ of $f$ have norm one, i.e., $|v_i|=1$ for all $i=1,\ldots,n$. An \struc{\emph{agiform}} is a special case of a circuit polynomial when choosing $f_{\boldsymbol{\alp}(j)}=\lambda_j$ and $f_{\boldsymbol{\beta}} =-1$, see \cite{Reznick:AGI}. We therefore limit our subsequent analysis to agiforms.  

As a first result in this context, we give an upper bound on the dimension of $C_{n,2d}(\Gamma)$.

\vspace*{0.15cm}
\begin{prop}
	\label{Prop:DimensionBoundExposedFaces}
	~\vspace*{-0.15cm}	\begin{enumerate}
		\item[$(1)$] Let $|\Gamma|=2^n$. Then $\dim(C_{n,2d}(\Gamma)) \leq \binom{n+d}{d}$. 
		\item[$(2)$] Let $|\Gamma|=2^{n-1}$. Then $\dim (C_{n,2d}(\Gamma)) \leq \binom{n+2d}{2d}$. 
	\end{enumerate}
\end{prop}	

\vspace*{0.15cm}
\begin{proof}
	~\vspace*{-0.15cm}\begin{enumerate}
		\item[(1)] Let $|\Gamma|=2^n$ and let $f$ be an agiform in $C_{n,2d}(\Gamma)$. By Theorem~\ref{Thm:ExactNumberAffineZeroes} $f$ must have an even inner exponent, so the whole support of $f$ is even. As already noted above, the dimension of $C_{n,2d}(\Gamma)$ is equal to the number of linear independent agiforms in $C_{n,2d}$ vanishing at all points of $\Gamma$. Thus,  $\dim(C_{n,2d}(\Gamma))$ is equivalent to the rank of the matrix $A \in \R^{m\times N(n,d)}$, where $m$ is the number of all agiforms in $C_{n,2d}$ vanishing at all points of $\Gamma$ and $\struc{N(n,d)}\coloneqq\binom{n+d}{d}$ is the number of all monomials $\mathbf{x}^{\boldsymbol{\alp}} \in \R[\mathbf{x}]_{n,2d}$ with even exponents $\boldsymbol{\alp}$. Since \linebreak $\rank(A)\leq \min\{m,N(n,d)\}$ and $m \geq N(n,d)$ for $2d\geq 6$ we conclude that the rank of $A$ is at most $\binom{n+d}{d}$. 
		
		\item[(2)] Now let $|\Gamma|=2^{n-1}$. Observe that the inner exponent of an agiform vanishing at all points of $\Gamma$ may be both even and odd. Therefore we have to take all monomials up to degree $2d$ into account, whereby the matrix $A$ from above is in $\R^{m\times N(n,2d)}$. The result now follows by similar arguments as for Statement (1). 
	\end{enumerate}
	%
	%
	\vspace{-\topsep}
\end{proof}

\vspace*{1mm}
Note that the dimension bound for $|\Gamma|=2^{n-1}$ is very naive and by comparison with $P_{n,2d}(\Gamma)$ also not likely to be sharp at all. In Section~\ref{SubSec:ExpFaces_ImprovedBound} we provide an improvement of this bound. It is a delicate question to exactly determine the number of linear independent agiforms in $C_{n,2d}$ vanishing at all points for a given $\Gamma$. Here we give a complete answer in the univariate case.

\subsection{The univariate Case}

In this subsection we study the univariate case. Since here we actually may count the dimension by hand, we initially compute $\dim (C_{1,2d}(\Gamma))$ for some small degree. Then, we determine the dimension of $C_{1,2d}(\Gamma)$ for general degree.\\

First, let $\Gamma = \{1,-1\}$, hence $|\Gamma|=2$. There is no agiform of degree $2$ vanishing at both $s\in \Gamma$, since the inner exponent cannot be even. The following table shows the calculation of the number of univariate linear independent agiforms vanishing at $1$ and $-1$ for the degree $4\leq 2d\leq 14$. We also compute $\dim(P_{1,2d}(\Gamma))$ for these cases:

\vspace*{0.4cm}
\begin{center}
	\begin{tabular}{c||C{1cm}|C{1cm}|C{1cm}|C{1cm}|C{1cm}|C{1cm}}
		2d 	 & $4$ & $6$ & $8$ & $10$ & $12$ & $14$  \\
		\hline
		\hline
		\rule{0pt}{15pt} $\dim (C_{1,2d}(\Gamma))$ & $1$ & $2$ & $3$ & $4$ & $5$ & $6$\rule{0pt}{15pt}\\
		\hline
		\rule{0pt}{15pt} $\dim(P_{1,2d}(\Gamma))$ & & $2$ & $5$ & $7$ & $9$ & $11$\rule{0pt}{15pt}
	\end{tabular}
\end{center}
\vspace*{0.4cm}

For degree $2d=4$ the dimension of $P_{1,4}(\Gamma)$ is omitted since it is unclear if the equality $\dim(P_{1,4}(\Gamma)) = \dim(I(\Gamma)_{2,4})$ holds in this case. 
Exemplary, we explain the calculations for $2d=8$. The dimension of $P_{1,8}(\Gamma)$ follows by \eqref{Eq:AlexHirsch}: $\dim(P_{1,8}(\Gamma))=\binom{8}{9}-2\cdot 2=5.$ Agiforms of degree $8$ vanishing on the given set have the form $\lambda_0+\lambda_1x^8-x^{\beta}$, where $\beta$ has to be even. Therefore, $\beta\in \{2,4,6\}$ and  $\dim (C_{1,8}(\Gamma))=3$.
Note that the first dimensional difference of the exposed faces of $C_{1,2d}$ and $P_{1,2d}$ is in degree~$8$. 

 We now consider $|\Gamma|=1$, i.e., $\Gamma=\{1\}$. Recall that we reduce our study to the case of agiforms, for which $-1$ is not a zero. This leads to the following dimensions:

\vspace*{0.4cm}
\begin{center}
	\begin{tabular}{c||C{1cm}|C{1cm}|C{1cm}|C{1cm}|C{1cm}|C{1cm}}
		2d	& $2$ & $4$ & $6$ & $8$ & $10$ & $12$  \\
		\hline
		\hline
		\rule{0pt}{15pt} $\dim (C_{1,2d}(\Gamma))$ & $1$ & $3$ & $5$ & $7$ & $9$ & $11$\rule{0pt}{15pt}\\
		\hline
		\rule{0pt}{15pt} $\dim(P_{1,2d}(\Gamma))$ &	 & $3$ & $5$ & $7$ & $9$ & $11$\rule{0pt}{15pt}
	\end{tabular}
\end{center}
\vspace*{0.4cm}

 Remember that $2d=2$ is one exception in the Alexander-Hirschowitz Theorem. 
Again we justify these computations for the degree $2d=8$. Equation \eqref{Eq:AlexHirsch} directly yields $\dim(P_{1,8}(\Gamma))=7$. Here, agiforms are of the same form as for the case $|\Gamma|=2$, only  with the difference, that $\beta$ might also be odd. Hence, $1\leq \beta\leq 7$, which leads to $\dim(C_{1,8}(\Gamma))=7$. Obviously, there is no dimensional gap between the dimensions of the exposed faces of $C_{1,2d}$ and $P_{1,2d}$ for the considered degrees.

\vspace*{2mm}
With these calculations in mind, we provide the exact dimension of $C_{1,2d}(\Gamma)$ in the general case of degree $2d$:
\vspace*{0.15cm}
\begin{lemma}
	~\vspace*{-0.15cm}\begin{enumerate}
		\item[$(1)$] For $|\Gamma|=2$, we have $\dim(C_{1,2d}(\Gamma))= d-1$, if $d\geq 2$. 
		\item[$(2)$] For $|\Gamma|=1$, we have $\dim(C_{1,2d}(\Gamma))= 2d-1$.
	\end{enumerate}
\label{Lem:DimExpFaceUnivariateCase}
\end{lemma}
\enlargethispage{\baselineskip}
\begin{proof}
	~\vspace*{-0.15cm}\begin{enumerate}
		\item[(1)] Let $|\Gamma|=2$. Clearly, in this case the number of linear independent agiforms with even inner exponent equals the number of even lattice points in $\Delta_{1,2d}$ without the vertices. This is equivalent to the number of all monomials $x^{\alp}$ in $\R[x]_{1,2d}$ with even degree $0<\alp<2d$. Hence, $\dim(C_{1,2d}(\Gamma))= \binom{n+d}{d}-2=d -1$. Here we have $d\geq 2$, because, as already noted, in dimension $2d=2$ there exists no agiform with even inner exponent.
		\item[(2)] Now let $|\Gamma|=1$. Analogously to case (1) the sought number equals the number of all monomials $x^{\alp}$ in $\R[x]_{1,2d}$ with $0<\alp<2d$. Thus, it follows immediately $\dim (C_{1,2d}(\Gamma))= \binom{n+2d}{2d}-2= 2d -1$.
	\end{enumerate}
	\vspace{-\topsep}
\end{proof}

\noindent The analysis of $\dim(P_{1,2d}(\Gamma))$ for both finite sets $\Gamma$ yields:
\begin{enumerate}
	\item[(1)] For $|\Gamma|=2$, we have $\dim(P_{1,2d}(\Gamma))= 2d-3$, if $d\geq 3$. 
	\item[(2)] For $|\Gamma|=1$, we have $\dim(P_{1,2d}(\Gamma))= 2d-1$, if $d\geq 2$. 
\end{enumerate}

\noindent Observe that, as indicated by our calculations by hand, we have indeed the equality $\dim(P_{1,2d}(\Gamma))=\dim(C_{1,2d}(\Gamma))$ if $|\Gamma|=1$. For $|\Gamma|=2$ on the contrary, the dimension of the exposed face $P_{1,2d}(\Gamma)$ is nearly twice as large as $\dim(C_{1,2d}(\Gamma))$.

\subsection{The bivariate Case}

We now examine the bivariate case. Again we calculate the dimensions of the exposed faces $C_{2,2d}(\Gamma)$ for some degrees. Already in this case we have to limit the calculation of explicit cases to the two smallest degrees due to the large number of agiforms in $C_{2,2d}$ even for low degrees. \\
For $|\Gamma|=4$, namely $\Gamma=\{(1,1),(-1,-1),(1,-1),(-1,1)\}$, we have:

\vspace*{0.2cm}
\begin{center}
	\begin{tabular}{c||C{1cm}|C{1cm}}
		2d	& $4$ & $6$   \\
		\hline
		\hline
		\rule{0pt}{15pt} $\dim (C_{2,2d}(\Gamma))$ & $3$ & $8$\rule{0pt}{15pt}\\
		\hline
		\rule{0pt}{15pt} $\dim(P_{2,2d}(\Gamma))$ &	 & $16$\rule{0pt}{15pt}
	\end{tabular}
\end{center}
\vspace*{0.4cm}

There is a noticeable dimensional difference between the exposed face of $C_{2,6}$ and $P_{2,6}$ already for $2d=6$. Moreover, observe that the exposed face $C_{2,6}(\Gamma)$ contains the agiform $$f(x_1,x_2)=\frac{1}{3}+ \frac{1}{3}x_1^4x_2^2 + \frac{1}{3}x_1^2x_2^4 - x_1^2x_2^2.$$
This polynomial can easily be detected to be one third of the Motzkin polynomial $g$, i.e., $\frac{1}{3}\cdot g = f$. Thus, we conclude $C_{2,6}(\Gamma)\not\subseteq \Sigma_{2,6}(\Gamma)$ which is in line with our knowledge regarding the cone containment of $C_{n,2d}$ and $\Sig_{n,2d}$.\\

For $|\Gamma|=2$, i.e., $\Gamma=\{(1,1),(-1,-1)\}$, we compute the following  dimensions:

\vspace*{0.2cm}
\begin{center}
	\begin{tabular}{c||C{1cm}|C{1cm}}
		2d	& $2$ & $4$   \\
		\hline
		\hline
		\rule{0pt}{15pt} $\dim (C_{2,2d}(\Gamma))$ & $1$ & $6$\rule{0pt}{15pt}\\
		\hline
		\rule{0pt}{15pt} $\dim(P_{2,2d}(\Gamma))$ &  & $9$\rule{0pt}{15pt}
	\end{tabular}
\end{center}
\vspace*{0.4cm}

Also here we may already detect in the first calculable case dimensional differences of the exposed faces of the analyzed cones. 

When counting the agiforms $f$ that vanish on $\Gamma$ in this case, we observe that not all $\binom{2+2d}{2d}$ possible monomials $\mathbf{x}^{\boldsymbol{\alp}} \in \R[\mathbf{x}]_{2,2d}$ appear in the support of $f$. Since $f$ has to vanish on both $\mathbf{s}\in \Gamma$ the inner exponent must have a special structure. 
To be more precise, for the inner exponent $\boldsymbol{\beta}$ of an agiform vanishing on $\Gamma=\{(1,1),(-1,-1)\}$ it must hold that $|\boldsymbol{\beta}|$ is even. Note that not necessarily each component of $\boldsymbol{\beta}$ has to be even, that is, it may hold that $\boldsymbol{\beta}\not\in (2\N)^n$. 

\vspace*{2mm}
Due to this observation, we can give a refined dimension bound of the exposed face $C_{2,2d}(\Gamma)$ in the case $|\Gamma|=2$ compared to Proposition~\ref{Prop:DimensionBoundExposedFaces} (2). 

\begin{lemma}
	For $|\Gamma|=2$, we have $\dim (C_{2,2d}(\Gamma))\leq d^2 +2d + 1$.
\end{lemma}

\begin{proof}
	The bound follows by counting the involved monomials in the support of the\linebreak agiforms vanishing on $\Gamma=\{(1,1),(-1,-1)\}$, which are all monomials $\mathbf{x}^{\boldsymbol{\alp}} \in \R[\mathbf{x}]_{2,2d}$ with $|\boldsymbol{\alp}|$ even.
\end{proof}

\subsection{Improved Dimension Bound}
\label{SubSec:ExpFaces_ImprovedBound}

If $n$ is even we can be even more precise, which leads to the following improved bound for $\dim (C_{n,2d}(\Gamma))$ when $|\Gamma|=2^{n-1}$:

\begin{prop}
	\label{Prop:DimensionBoundExposedFacesRefined}
	Let the number of variables $n$ be even and $|\Gamma|=2^{n-1}$. Then \linebreak $\dim (C_{n,2d}(\Gamma)) \leq \sum_{i=0}^{d}\binom{n+2i-1}{2i}$. 
\end{prop}

\begin{proof}
	The dimension of $C_{n,2d}(\Gamma)$ is bounded by the number of all monomials of degree at most $2d$ with even degree. 
	The number of monomials $\mathbf{x}^{\boldsymbol{\alp}} \in \R[\mathbf{x}]_{n,2d}$ having exactly even degree $2i$, i.e., $|\boldsymbol{\alp}|=2i$, is given by $\binom{n+2i-1}{2i}$. Hence, summing over all $i=0,\ldots,d$ yields the right number.
\end{proof}
This argumentation does not hold for $n$ odd, since the inner monomial $\mathbf{x}^{\boldsymbol{\beta}}$ of an agiform vanishing on all $\mathbf{s}\in \Gamma$ with $|\Gamma|=2^{n-1}$ may also have an odd degree. For instance, consider the $3$-variate case and $\Gamma=\{(1,1,1),(1,-1,-1),(-1,1,-1),(-1,-1,1)\}$. The agiform $f=\frac{1}{4}+\frac{1}{4} x_1^4+\frac{1}{4}x_2^4+\frac{1}{4}x_3^4-x_1x_2x_3$ vanishes on every $\mathbf{s}\in \Gamma$, but $|\boldsymbol{\beta}|=3$.

Furthermore, observe that for $n\geq 4$ the bound in Proposition~\ref{Prop:DimensionBoundExposedFacesRefined}
can be further refined. In the occurring sum we also count monomials $\mathbf{x}^{\boldsymbol{\alp}}$ with $\alpha_i=0$ for some $i$. For example if $n=4$ we also take the monomials $x_1^3x_2$, $x_2x_3$, or $x_2x_3x_4^2$ into account. But clearly, these monomials cannot be inner terms of an agiform vanishing on all $\mathbf{s}\in \Gamma$, with $|\Gamma|=8$.

\medskip


\section{Conclusion and Outlook}

This paper provides a complete classification of the real zeros of SONC polynomials and forms yielding some interesting additional results. Using the observations of the real zeros, we initiate the analysis of the exposed faces $C_{n,2d}(\Gamma)$ and provide initial dimension bounds. It would be an interesting task to further improve the dimension bounds (also in the case of an odd number of variables) or actually determining precisely the dimension of the exposed faces of $C_{n,2d}$. Moreover, the gaps between the dimensions of the exposed faces of $C_{n,2d}$ and $P_{n,2d}$ need to be explored in more detail. 

The recent work by Forsg\r{a}rd and de Wolff \cite{Forsgard:deWolff:SONC-Boundary} characterizes the algebraic boundary of the SONC cone. Among other results, the authors provide a description of the semialgebraic stratification of the boundary in the univariate case. The stratification depends on the common zeros of the involved circuit polynomials. We hope to explore possible connections to our results in the future.

\medskip

\bibliographystyle{amsalpha}
\bibliography{main}
	
\end{document}